\documentclass[11pt]{amsart}
\usepackage{pb-diagram,amssymb,epic,eepic,verbatim,graphicx}
\usepackage{graphics,epsfig,psfrag} 

\newtheorem{theorem}{Theorem}[section]

\newtheorem{corollary}[theorem]{Corollary}

\newtheorem{proposition}[theorem]{Proposition}
\newtheorem{lemma}[theorem]{Lemma}
\newtheorem{lem}[theorem]{}
\theoremstyle{definition}
\newtheorem{definition}[theorem]{Definition}
\theoremstyle{remark}
\newtheorem{remark}[theorem]{Remark}
\newtheorem{example}[theorem]{Example}
\newcommand{\blem}{\begin{lem} \rm}
\newcommand{\elem}{\end{lem}}

\newcommand\A{\mathcal{A}}
\newcommand\M{\mathcal{M}}

\renewcommand\M{\mathcal{M}}
\renewcommand\S{\mathcal{S}}
\newcommand{\K}{\mathcal{K}}

\renewcommand{\O}{\mathcal{O}}

\newcommand{\J}{\mathcal{J}}

\newcommand{\U}{\mathcal{U}}

\newcommand{\R}{\mathbb{R}}
\renewcommand{\H}{\mathbb{H}}

\newcommand{\C}{\mathbb{C}}

\newcommand{\Y}{\mathcal{Y}}
\newcommand{\Z}{\mathbb{Z}}
\newcommand{\Q}{\mathbb{Q}}

\newcommand{\ddt}{\frac{d}{dt}}

\renewcommand{\P}{\mathbb{P}}

\newcommand\lie[1]{\mathfrak{#1}}
\renewcommand{\k}{\lie{k}}
\renewcommand{\l}{\lie{l}}

\newcommand{\g}{\lie{g}}

\renewcommand{\r}{\lie{r}}

\newcommand{\z}{\lie{z}}
\renewcommand{\t}{\lie{t}}

\newcommand{\on}{\operatorname}

\newcommand{\vir}{\on{vir}}
\newcommand{\Def}{\on{Def}}
\newcommand{\Sym}{\on{S}}

\newcommand{\Ob}{\on{Ob}}

\newcommand{\fr}{{\on{fr}}}

\newcommand{\dual}{\vee}

\newcommand{\Aut}{ \on{Aut} } 
 
\newcommand{\aut}{ \on{aut} } 
\newcommand{\Ad}{ \on{Ad} }

\newcommand{\Hom}{ \on{Hom}}
\newcommand{\Ind}{ \on{Ind}}
\renewcommand{\ker}{ \on{ker}}

\newcommand{\Vol}{  \on{Vol}}

\newcommand{\codim}{\on{codim}}

\newcommand\dirac{/\kern-1.2ex\partial} 
\newcommand\qu{/\kern-.7ex/} 
\newcommand\lqu{\backslash \kern-.7ex \backslash} 

\newcommand\dr{r_+ \kern-.7ex - \kern-.7ex r_-}
 
\newcommand{\labell}\label

\renewcommand{\d}{{\on{d}}}
\newcommand{\ol}{\overline}
\newcommand{\olp}{\ol{\partial}}

\newcommand\phinv{\phi^{-1}}

\newcommand\eps{\epsilon}

\newcommand{\meps}{{\eps}}

\newcommand{\f}{\frac}
\newcommand{\lan}{\langle}
\newcommand{\ran}{\rangle}
\newcommand{\hh}{{\f{1}{2}}}

\newcommand{\ti}{\tilde}

\newcommand\pt{\on{pt}}

\newcommand\cF{\mathcal{F}}
\newcommand\cN{\mathcal{N}}

\newcommand\cH{\mathcal{H}}

\renewcommand{\ss}{{\on{ss}}}

\newcommand\mE{\mathcal{E}}

\newcommand\Gr{\on{Gr}}
\newcommand\Map{\on{Map}}
\newcommand\rank{\on{rank}}
\newcommand\ev{\on{ev}}
\newcommand\Eul{\on{Eul}}

\newcommand\Vect{\on{Vect}}
\newcommand\ul{\underline}
\newcommand\mO{\mathcal{O}}

\renewcommand\H{\mathcal{H}}
\newcommand\cZ{\mathcal{Z}}
\newcommand\cY{\mathcal{Y}}
\renewcommand\Im{\on{Im}}
\newcommand\Ker{\on{Ker}}

\newcommand\reg{{\on{reg}}}
\newcommand\mdeg{\mu}

\newcommand\bra[1]{ < \kern-.7ex {#1} \kern-.7ex >} 
\newcommand\bdefn{\begin{definition}}
\newcommand\edefn{\end{definition}}
\newcommand\bea{\begin{eqnarray*}}
\newcommand\eea{\end{eqnarray*}}
\newcommand\bcv{\left[ \begin{array}{r} }
\newcommand\ecv{\end{array} \right] }

\newcommand\bma{\left[ \begin{array}{l} }
\newcommand\ema{\end{array} \right]}
\newcommand\ben{\begin{enumerate}}
\newcommand\een{\end{enumerate}}
\newcommand\beq{\begin{equation}}
\newcommand\eeq{\end{equation}}
\newcommand\bex{\begin{example}}
\newcommand\bsj{\left\{ \begin{array}{rrr} }
\newcommand\esj{\end{array} \right\}}

\newcommand\Ch{\on{Ch}}

\newcommand\eex{\end{example}}

\newcommand\sx{*\kern-.5ex_X}

\def\mathunderaccent#1{\let\theaccent#1\mathpalette\putaccentunder}
\def\putaccentunder#1#2{\oalign{$#1#2$\crcr\hidewidth \vbox
to.2ex{\hbox{$#1\theaccent{}$}\vss}\hidewidth}}

\begin{document}

\title{Gauged Gromov-Witten theory for small spheres}

\author{Eduardo Gonzalez}

\address{
Department of Mathematics
University of Massachusetts Boston
100 William T. Morrissey Boulevard
Boston, MA 02125}
  \email{eduardo@math.umb.edu}

\author{Chris Woodward}

\address{Mathematics-Hill Center,
Rutgers University, 110 Frelinghuysen Road, Piscataway, NJ 08854-8019,
U.S.A.}  \email{ctw@math.rutgers.edu}

\thanks{Partially supported by NSF grants DMS1104670, DMS0904358
  and the Simons Center for Geometry and Physics.  To appear in Math. Zeit.  Final version available at www.springerlink.com}

\begin{abstract}  
We relate the genus zero gauged Gromov-Witten invariants of a smooth
projective variety for sufficiently small area with equivariant
Gromov-Witten invariants.  As an application we deduce a gauged
version of abelianization for Gromov-Witten invariants in the small
area chamber.  In the symplectic setting, we prove that any sequence
of genus zero symplectic vortices with vanishing area has a
subsequence that converges after gauge transformation to a holomorphic
map with zero average moment map.
\end{abstract}

\maketitle


\section{Introduction}

Gauged Gromov-Witten invariants are generalizations of usual
Gromov-Witten invariants to the case of symplectic manifolds with
group action.  These invariants are defined as integrals over moduli
spaces of gauged versions of holomorphic maps, known in the
symplectic approach as symplectic vortices, and depend on a choice of
area form on the domain curve.  In this paper we study the limit of
the gauged Gromov-Witten invariants in the case of genus zero curves
with area tending to zero.

Ideally one would like to define gauged Gromov-Witten invariants for
arbitrary Hamiltonian group actions.  Let $K$ be a compact Lie group,
$X$ a Hamiltonian $K$-manifold equipped with a moment map $\Phi: X \to
\k^\dual$, an invariant compatible almost complex structure $J$, and
let $\Sigma$ be a compact smooth complex curve.  A {\em gauged
  holomorphic map} from $\Sigma$ to $X$ is a pair $(A,u)$ consisting
of a connection $A$ on a principal $K$-bundle $P \to \Sigma$ together
with a section $u$ of the associated fiber bundle $P(X) := P \times_K
X$ that is holomorphic with respect to the complex structure
determined by $J$ and $A$.  The space of gauged holomorphic maps with
bundle $P$ has a formal Hamiltonian action of the group of gauge
transformations $\K(P)$.  The moment map depends on a choice of area
form $\Vol_\Sigma \in \Omega^2(\Sigma)$ and an invariant inner product
$( \ , \ ): \k \times \k \to \R$ on the Lie algebra $\k$ inducing an
identification $\k \to \k^\dual$.  The symplectic quotient of the
space of gauged holomorphic maps by the group of gauge transformations
is the moduli space of {\em symplectic vortices} \cite{ciel:vor},
\cite{ga:gw}, \cite{ci:symvortex}, \cite{mun:ham}
\begin{equation} \label{veq}
M^K(P,X) : = \{ (A,u) \, | \, F_A + u^* P(\Phi) \Vol_\Sigma = 0 \} / \K(P) ,
 \end{equation}
where $F_A \in \Omega^2(\Sigma,P(\k))$ is the curvature of $A$ and
$P(\Phi): P(X) \to P(\k^\dual) \cong P(\k)$ the map induced by $\Phi$.
Let $M^K(\Sigma,X)$ be the union of $M^K(P,X)$ over topological types
of bundles $P \to \Sigma$, and $M^K(\Sigma,X,d)$ the subspace of
homology class $d \in H_2^K(X,\Z)$.  In the case that $X$ is compact,
one obtains a natural compactification $\ol{M}^K(\Sigma,X,d)$ by
allowing bubbling of $u$ in the fibers of $P(X)$, so that $u$ is a
stable map to $P(X)$ as in Mundet \cite{mun:ham} and Ott
\cite{ott:remov}; we call such a nodal vortex $(A,u)$ {\em polystable}
and {\em stable} if it has finite automorphism group.  Gauged
Gromov-Witten invariants should be defined by virtual integration over
the moduli space $\ol{M}^K(\Sigma,X,d)$. 

The necessary virtual fundamental cycles are more easily defined for
algebraic targets.  Let $G$ be the complexification of $K$ and suppose
that $X$ is a smooth projectively-embedded $G$-variety.  If $P \to
\Sigma$ is a principal $K$-bundle then any connection $A$ on $P$
defines a holomorphic structure on the $G$-bundle $P(G) = P \times_K
G$ associated to $P$.  If $u$ is a holomorphic section of the
associated $X$ bundle $P(X) = P(G) \times_G X$ then the pair
$(P(G),u)$ is by definition a morphism from $\Sigma$ to the {\em
  quotient stack} $X/G$.  By Mundet's correspondence,
\cite{mund:corr}, an irreducible pair $(A,u)$ is
complex-gauge-equivalent to a vortex iff the corresponding holomorphic
pair $(P(G),u)$ satisfies a semistability condition generalizing
Mumford-Seshadri semistability for vector bundles on curves.  As
before, allowing bubbling in the fibers gives rise to a {\em moduli
  stack} $\ol{\M}_n^G(\Sigma,X)$, compact once the homology class is
fixed.  If semistable gauged maps have finite stabilizer then
$\ol{\M}_n^G(\Sigma,X)$ is a Deligne-Mumford stack whose underlying
coarse moduli space is homeomorphic to $\ol{M}_n^K(\Sigma,X)$.  The
machinery of Behrend-Fantechi \cite{bf:in} yields virtual fundamental
classes and the desired invariants.

The choice of area form gives rise a one-parameter family of Mundet
semistability conditions.  In order to study the dependence of the gauged
Gromov-Witten invariants on this choice we introduce a real parameter
$\rho$ and say that a {\em $\rho$-vortex} is a pair $(A,u)$ satisfying
the vortex equations with area form $\rho \Vol_\Sigma$, or
$\rho$-semistable if it is Mundet semistable with this choice.  Let
$\ol{M}_n(\P^1) := \ol{M}_{0,n}(\P^1,[\P^1])$ be the Fulton-MacPherson
moduli space of stable maps to $\P^1$ of homology class $[\P^1]$.  The
{\em gauged Gromov-Witten invariants}
$$ \lan \ , \ \ran_{d,\rho} : H_G(X,\Q)^n \otimes
H(\ol{M}_n(\P^1)) \to \Q $$
for $n \in \Z_{\ge 0}, d \in H_2(X,\Z)$ are defined by virtual
integration.  The appearance of the Fulton-MacPherson, rather than the
usual Grothendieck-Knudsen moduli space, is a manifestation of the
dependence of the moduli space on the choice of area form.  As the
stability parameter $\rho$ varies the gauged Gromov-Witten invariants
are related by a wall-crossing formula \cite{cross}.  There are two
interesting limits in which the area of the domain tends to infinity
or zero.  The former has been studied by Gaio-Salamon \cite{ga:gw}.
In certain cases, they show that the limit of any sequence of
solutions is a holomorphic map to the symplectic quotient $X \qu K$
and so obtain a relationship between the equivariant cohomology of $X$
and the quantum cohomology of $X \qu K$.

The purpose of this paper is to study the opposite limit in which the
area of the domain tends to zero.  The space of holomorphic maps
$\Hom(\P^1,X)$ from $\P^1$ to $X$ has a natural $K$-action given by composition
with the action, and a formal symplectic structure given by
integrating the pull-back of the symplectic form on $X$.  A formal
moment map is
$$ \phi: \Hom(\P^1,X) \to \k^\dual, \ \ \ u \mapsto \int_{\P^1} u^*
\Phi \Vol_{\P^1}.$$
The map $\phi$ extends to the moduli space of parametrized stable maps
$\ol{M}_{0,n}(\P^1 \times X, (1,d))$, that is the moduli of
$n$-marked stable maps of genus $0$ with degree $(1,d)$ into
$\P^1\times X$ where $1 = [\P^1]$ is the unit in $H_2(\P^1)$.  We say
that a {\em polystable zero-area vortex} of homology class $d \in
H_2(X,\Z)$ is a stable map ${u}: {C} \to \P^1 \times X$ from a genus
zero nodal curve $C$ of class $(1,d)$ with $\phi({u})= 0$.  The first
main result says that these are exactly the stable maps that appear in
the small area limit:
\begin{theorem}  \label{limthm} 
Suppose that $X$ is a compact Hamiltonian $K$-manifold, $\rho_\nu \to
0 $ is a sequence of positive real numbers, and $(A_\nu,{u}_\nu)$ is a
sequence of polystable $\rho_\nu$-vortices on $\P^1$ with target $X$
of fixed homology class and bundle $P$.  Then $P$ is trivializable and
there exists a sequence $k_\nu \in \K(P)$ of gauge transformations
such that after passing to a subsequence, $k_\nu A_\nu$ converges in
$C^0$ to the trivial connection on $P$ and $k_\nu {u}_\nu$ Gromov
converges to a polystable zero-area vortex ${u}_0$.  Conversely, any
regular stable zero-area vortex is a limit of a sequence
$(A_\nu,{u}_\nu)$ of $\rho_\nu$-vortices with $\rho_\nu \to 0$.
\end{theorem}
\noindent The fact that the limit is a holomorphic map should not be
surprising: given a sequence of vortices with vanishing area, the
curvature of the connection goes to zero.  Therefore, if the sequence
of connections converges then it must converge to a flat connection,
which in genus zero is trivializable.  From the symplectic point of
view, the subtle part of the theorem is the statement that the
limiting map has zero average moment map: this follows from a study of
a sub-leading term in the vortex equations.  A similar theorem for the
case that $X$ is aspherical and $K$ is abelian is proved in the thesis
of Jan Wehrheim \cite{jw:vi}.

The moduli space of $n$-marked polystable zero-area vortices is
denoted $\ol{M}_n^K(\P^1,X)_0$.  If $X$ is a smooth
projectively-embedded $G$-variety then $\ol{M}_n^K(\P^1,X)_0$ is the
coarse moduli space of a {\em moduli stack} $\ol{\M}_n^G(\P^1,X)_0$.
Integration gives rise to {\em zero-area gauged Gromov-Witten
  invariants} $ \lan \ , \ \ran_{d,0}$.  The zero-area gauged
Gromov-Witten invariants may be viewed as the {\em invariant part} of
Givental's equivariant Gromov-Witten invariants \cite{gi:eq}, and
standard techniques \cite{je:lo1}, \cite{wo:norm} allow, in principle,
their computation from Givental's invariants, as in Theorem
\ref{nonab} below.  The second main result is a description of the
gauged Gromov-Witten invariants in the small-area limit:

\begin{theorem}  \label{equal}  Suppose that $X$ is a smooth projectively-embedded variety
and $d \in H_2^G(X,\Z)$, and every $0$-semistable gauged map has
finite automorphism group.  There exists a $ \rho_0 > 0$ such that for
all $\rho < \rho_0$, there is an equivalence $\ol{\M}_n^G(\P^1,X,d)_0
\to \ol{\M}_n^G(\P^1,X,d)_\rho$ of Deligne-Mumford stacks with
relative perfect obstruction theories and an equality of invariants
$ \lan \alpha ; \beta \ran_{d,\rho} = \lan \alpha; \beta \ran_{d,0}$
for all $\alpha \in H_G(X,\Q)^n$ and $\beta \in
H(\ol{M}_n(\P^1),\Q)$.
\end{theorem} 

Using this result we obtain a version of the abelianization conjecture
of Bertram, Ciocan-Fontanine, and Kim \cite{be:qu} in the ``small area
chamber'' relating the Gromov-Witten invariants for $G$ and a maximal
torus $T$, in the case that $G$ is connected.  Let $W$ be the Weyl
group of $T \subset G$.

\begin{theorem}  \label{qmartinthm}
   Let $X$ be a smooth projectively-embedded $G$-variety such that
   every $0$-stable gauged map has finite automorphism group and $d_G
   \in H_2^G(X,\Z)$.  There exists a constant $\rho_0 > 0$ such that
   if $\rho < \rho_0$ and $\alpha \in H_G(X,\Q)^n$ is an equivariant
   Hodge class (that is, a Chern character of an equivariant algebraic
   vector bundle) and $\beta \in H(\ol{M}_n(\P^1),\Q)$ then
\begin{equation} \label{qmartininfty}
 \lan \alpha, \beta \ran_{G,d_G,\rho} = (\# W)^{-1} \sum_{d_T \mapsto
   d_G} \lan \alpha , \beta \ran_{T,d_T,\rho}^{twist} \end{equation}
where the right hand side is a sum of gauged Gromov-Witten invariants
twisted by the bundle with fiber $\g/\t$, where $\g$ is the Lie
algebra of $G$ and $\t \subset \g$ a Cartan subalgebra.  All of these
gauged Gromov-Witten invariants vanish unless $d_G$ is in the image of
$H_2(X)$ in $H_2^G(X)$.
\end{theorem} 

Combining this with the wall-crossing formula of \cite{cross} gives an
abelianization formula for all chambers, in particular, the large area
chamber which is related to the Gromov-Witten theory for the
symplectic quotient.  

A word on notation: we generally use the notation 
${M}$ for a
coarse moduli space (space of isomorphism classes), ${\M}$ for a
Deligne-Mumford stack, and ${\frak{M}}$ for an Artin stack, so the
``ornateness'' of the notation corresponds to ``stackiness''.

\section{Symplectic vortices and gauged Gromov-Witten invariants}

In this section we recall the moduli spaces of symplectic vortices
associated to Hamiltonian $K$-manifolds, introduced by Mundet
\cite{mun:ham} and Salamon and collaborators \cite{ci:symvortex},
\cite{ciel:vor}, and the associated gauged Gromov-Witten invariants
introduced (in the case of algebraic target) in \cite{cross}.

\subsection{Symplectic vortices}

Let $\Sigma$ be a compact smooth complex curve with complex structure
$J_\Sigma: T\Sigma \to T\Sigma$, and $\pi: P \to \Sigma$ a smooth
principal $K$-bundle.  Let $X$ be a compact Hamiltonian
$K$-manifold with symplectic form $\omega$ and moment map $\Phi:X \to
\k^\dual$.  The action of $K$ on $X$ gives rise to a homomorphism of Lie
algebras 
$$\k \to \Vect(X), \ \ \  \xi \mapsto \xi_X, 
\quad  \xi_X(x) = \ddt |_{t=0} \exp(-t \xi) x .$$
Our sign convention for the moment map is 
$\iota(\xi_X) \omega = - \d \lan \Phi,\xi \ran, \forall \xi \in \k. $
By equivariant formality of Hamiltonian $K$-manifolds (see
e.g. \cite{gu:eqdr}) the second equivariant homology splits $
H_2^K(X,\Q) \cong H_2(X,\Q) \oplus H_2^K(\pt,\Q) .$ If the action of
$K$ extends to the complexification $G$ then we identify $H^K(X)$ and
$H^G(X)$; generally speaking we use $H^G(X)$ when $X$ is a
$G$-manifold, or $H^K(X)$ when we are using only the $K$-action.
Continuous sections $u$ of $P(X) := (P \times X)/K$ correspond to
lifts $u_K: \Sigma \to X_K = EK \times_K X$ of a classifying map
$\Sigma \to BK$ for $P$ to $X_K$.  Denote by $\pi_{P(X)}$ the
projection of $P(X)$ onto $\Sigma$.  The {\em equivariant homology
  class} $\mdeg(u) \in H_2^K(X,\Z)$ of $u: \Sigma \to P(X)$ is defined
by $\mdeg(u) = u_{K,*}[\Sigma]$.  Let $\J(X)$ denote the space of
compatible almost complex structures on $X$. The action of $K$ induces
an action on $\J(X)$ by conjugation, and we denote by $\J(X)^K$ the
invariant subspace.  Let $\A(P)$ denote the space of smooth
connections on $P$, and by $P(\k) := (P \times \k)/K$ the adjoint
bundle.  The space $\A(P)$ is an affine space modelled on
$\Omega^1(\Sigma,P(\k))$, with action given by $A \mapsto A + \pi^* a$
where $\pi^*: \Omega^1(\Sigma,P(\k)) \to \Omega^1(P,\k)^K$ is the
pull-back map.  For any $A \in \A(P)$, we denote $F_A \in
\Omega^2(\Sigma,P(\k))$ the curvature of $A$.  Any connection $A \in
\A(P)$ induces a map of spaces of almost complex structures
$$ \J(X)^K \to \J(P(X)), \ \ J \mapsto J_A$$
 using the splitting defined by the connection.  Let
 $\Gamma(\Sigma,P(X))$ denote the space of smooth sections of $P(X)$.
 Let
$$ \olp_A : \Gamma(\Sigma,P(X)) \to \Omega^{0,1}(\Sigma,( \cdot
)^*T^{\on{vert}} P(X)), \quad  \olp_A u = \hh ( \d u + J_A(u) \circ \d u \circ J_\Sigma) $$
be the Cauchy-Riemann operator defined by $J_A$.  A {\em gauged map to
  $X$} is a pair $(A,u) \in \A(P) \times \Gamma(\Sigma,P(X))$.
Suppose $\Vol_\Sigma$ is the area form determined by a choice of
metric on $\Sigma$.  The {\em energy} of a gauged section $(A,u)$ is
given by
$$ E(A,u) = \hh \int_\Sigma \left(| \d_A u |^2 + |F_A|^2 + |u^*
 P(\Phi)|^2 \right) \Vol_\Sigma .$$
The \emph{equivariant area} $D(u)$ of $u$ is pairing of the class
$\mdeg(u)$ with the class $[\omega_K] \in H^2_K(X)$ of the equivariant
symplectic form.  The energy and equivariant area are related by
\cite[3.1]{ci:symvortex}
\begin{equation}  \label{energyaction} 
E(A,u) = D(u) + \int_\Sigma \left( | \olp_A u |^2 + \hh | F_A + u^*
P(\Phi) \Vol_\Sigma |^2 \right) \Vol_\Sigma .
\end{equation}
The space of {\em gauged holomorphic maps} with underlying bundle $P$
is
$$ \H(P,X) = \{ (A,u) \in \A(P) \times \Gamma(\Sigma,P(X)), \ \ \olp_A
u 
= 0 \} .$$
The space $\H(P,X)$ has a formal symplectic form induced from the sum
of the formal symplectic forms on the factors, given as follows.  Let
$$ \Omega^1(P(\k))^2 \to \R, \ \ \ (a_1,a_2) \mapsto \int_{\Sigma} (a_1
\wedge a_2 ) $$
denote the symplectic form on the affine space of connections $\A(P)$
determined by the metric on $\k$.  On the other hand, let $P(\omega)$
denote the fiber-wise two-form on $P(X)$ defined by $\omega$.  Choose
a two-form $\Vol_{\Sigma} \in \Omega^2(\Sigma)$ and define
$$ \Omega^0(\Sigma, u^* T^{\on{vert}}P(X))^2 \to \R, \ \ \ (v_1,v_2)
\mapsto \int_{\Sigma} (u^* P(\omega))(v_1,v_2) \Vol_{\Sigma} .$$
Choose a constant $\rho > 0$, called the {\em vortex parameter},
and consider the formal two-form
\begin{equation} \label{sympform}
 ((a_1,v_1),(a_2,v_2)) \to \int_{\Sigma} ( a_1 \wedge
a_2 ) + \rho \Vol_{\Sigma} (u^* P(\omega))(v_1,v_2) \end{equation} 
Let $\K(P)$ denote the group of gauge transformations.  The Lie
algebra $\k(P)$ of $\K(P)$ is the space of sections
$\Omega^0(\Sigma,P(\k))$ of the adjoint bundle.  The action of $\K(P)$
on the space of pairs generates an infinitesimal action
$ (A,u) \mapsto (  \d_A \xi, u^* \xi_X) $
where $\xi_X \in \Omega^0(\Sigma,P \times_K \Vect(X))$ is the vertical
vector field induced by $\xi$.  The action preserves the formal
two-form \eqref{sympform} and has moment map given by the curvature
plus pull-back of the moment map for $X$,
$$ \A(P) \times \Gamma(\Sigma,P(X)) \to \Omega^2(\Sigma,P(\k)),
\ \ \ (A,u) \mapsto F_{A,u} := F_A + \rho \Vol_{\Sigma} u^* P(\Phi)
.$$
By restriction one obtains a formal Hamiltonian action on $\H(P,X)$.
These formal considerations motivate the following definition.
\begin{definition} \label{vdef} A
 gauged holomorphic map $(A,u)$ is a {\em $\rho$-vortex}
 iff $ F_{A,u} = 0 .$ \end{definition}
\noindent 
The {\em moduli space of
  $\rho$-vortices} of class $d \in H_2^K(X)$ with bundle $P$ is 
$$ M^K(P,X,d)_\rho 
:=  \{ F_{A,u} = 0, \mdeg(u) =d \}/\K(P) .$$
Let $M^K(P,X)_\rho$ be the union over classes $d$, and
$M^K(\Sigma,X)_\rho$ the union over types $P$.  Note that the first
Chern class of $P$ is determined by the homology class of $u_K$ via
the projection $X_K \to BK$.  The formal tangent space to
$M^K(P,X)_\rho$ is the kernel of a Fredholm operator given as follows.
We first give the spaces of connections and sections the structure of
Banach manifolds by taking completions with respect to Sobolev norms
$\| \cdot \|_{k,p}$
for positive integers $k,p$. For $p>2$, define
\begin{multline*} \d_{A,u,\rho}: \Omega^1(\Sigma,P(\k))_{1,p} \oplus 
\Omega^0(\Sigma,u^* T^{\on{vert}} P(X) )_{1,p} \to
\Omega^2(\Sigma,P(\k))_{0,p} \\ (a,v) \mapsto \d_A a + \rho
\Vol_{\Sigma} L_v P(\Phi)
\end{multline*}
\begin{multline*} \d^*_{A,u,\rho}: \Omega^1(\Sigma,P(\k))_{1,p} \oplus
\Omega^0(\Sigma,u^* T^{\on{vert}}P(X) )_{1,p} \to
\Omega^0(\Sigma,P(\k))_{0,p} \\
(a,v) \mapsto \d_A^* a + \rho \Vol_{\Sigma} u^* L_{J v} P(\Phi)
.\end{multline*}
Here $L_v P(\Phi) \in \Omega^0(\Sigma,P(\k))$ denotes the derivative
of $P(\Phi)$ in the direction of $v$.  The set
\begin{equation} \label{slice} S_{A,u} = \{ (A + \pi^* a,
  \exp_u(v)), (a,v) \in \ker \d_{A,u,\rho}^*\} \subset
  \H(P,X)_{1,p}
\end{equation}
contains a slice for the action of $\K(P)_{2,p}$ on $\H(P,X)_{1,p}$ at
$(A,u)$ (see \cite[Section 4]{ci:symvortex}).  Define
\begin{multline} \label{cutout}
 \cF_{A,u}^\rho : \Omega^1(\Sigma,P(\k))_{1,p} \oplus
\Omega^0(\Sigma,u^* T^{\on{vert}} P(X))_{1,p} \\ \to (\Omega^0 \oplus
\Omega^2)(\Sigma,P(\k))_{0,p} \oplus \Omega^{0,1}(\Sigma,u^*
T^{\on{vert}} P(X))_{0,p} \\
 (a,v) \mapsto \left( F_{A + \pi^*a} + \rho \Vol_\Sigma
  \exp_u(v)^* \Phi, \d_{A,\rho}^* (a,v), \Psi_u(v)^{-1}
  \ol{\partial}_{A + \pi^*a} \exp_u(v) \right) \end{multline}
where $\Psi_u(v)^{-1}$ is parallel transport from $\exp_u(v)$ to
$u$ using the Hermitian modification of the Levi-Civita connection
\cite[Section 3.1]{ms:jh}.  Let $D_{A,u}$ denote the linearization of
the Cauchy-Riemann operator for $J_A$,
$$ D_{A,u} (a,v) := (\nabla_A v)^{0,1} + a_X^{0,1} + \hh J_u
(\nabla_{A,v} J )_u \partial_A u $$
where $a_X$ denotes the image of $a$ under the map
$\Omega^1(\Sigma,P(\k)) \to \Omega^1(\Sigma,u^* T^{\on{vert}} P(X))$
induced by the action, and $0,1$ denotes projection on the
$0,1$-component.  The last term vanishes if $J$ is integrable.

\begin{definition} The {\em linearized operator} for $(A,u) \in
\H(P,X)$ and vortex parameter $\rho$ is the operator
\begin{multline}  \label{linearized}
 \ti{D}^\rho_{A,u} :\Omega^1(\Sigma,P(\k)) \oplus \Omega^{0}(\Sigma, u^*
 T^{\on{vert}}(P(X)) \\ \to (\Omega^0 \oplus \Omega^2)(\Sigma, P(\k))
 \oplus \Omega^{0,1}(\Sigma, u^* T^{\on{vert}}(P(X)) \\ (a,v)
 \mapsto (\d_{A,u,\rho}( a,v), \d_{A,u,\rho}^* (a,v) ,
 D_{A,u}(a,v) ).
\end{multline}
$(A,u)$ is {\em regular} if the operator $\ti{D}^\rho_{A,u}$ is
surjective.  If $(A,u)$ is regular then it follows from the slice
condition that $\aut(A,u)$ is trivial, and since $\Aut(A,u) \subset
\Aut(A)$ is compact, $\Aut(A,u)$ is finite, that is, $(A,u)$ is
stable.  The space of {\em infinitesimal deformations} of $(A,u)$ is
$\Def(A,u) = \ker(\ti{D}^\rho_{A,u})$.  \end{definition}
\noindent The operator $\ti{D}^\rho_{A,u}$ is the linearization of the
map $ \cF_{A,u}^\rho$ at $(A,u)$.  It is elliptic, and so has finite
dimensional kernel and cokernel in the Sobolev completions above.  The
following theorem, due to Mundet, Salamon et al. \cite{mun:ham},
\cite{ci:symvortex}, generalizes the standard results for
pseudoholomorphic maps to the gauged setting:

\begin{theorem}  \label{smoothreg} 
For any constants $c_1,c_2 > 0$, the set of elements $[A,u] \in
M^K(\Sigma,X)_\rho$ with
$\sup | \d_A u | < c_1$ and $E(A,u) < c_2$
is compact.  The regular locus $M^{K,\reg}(\Sigma,X)_\rho$ is
a smooth orbifold with tangent space 
at $[A,u]$ 
isomorphic to
$\Def(A,u)$. The dimension of the component of homology class $d \in
H_2^K(X)$ is given by
$$ \dim(M^{K,\reg}(\Sigma,X,d)_\rho) = \Ind(\ti{D}^\rho_{A,u}) =
(1-g)(\dim(X) - 2\dim(K)) + 2(c_1^K(TX),d) $$
where $g = \on{genus}(\Sigma)$. 
\end{theorem}  

Spaces with markings and framings are given as follows.
\begin{definition}  
An {\em $n$-marked} symplectic vortex is a vortex $(A,u)$ together
with $n$-tuple $\ul{z} = (z_1,\ldots, z_n)$ of distinct points on
$\Sigma$.  An {\em isomorphism} of $n$-marked symplectic vortices is
an isomorphism of the underlying vortices, such that the markings are
equal.  A {\em framed vortex} is a collection
$(A,u,\ul{z},\ul{\phi})$, where $(A,u,\ul{z})$ is a marked
vortex and $\ul{\phi} = (\phi_1,\ldots,\phi_n)$ are
trivializations of the fibers of $P$ at $z_1,\ldots, z_n$, that is,
each $\phi_j: P_{z_j} \to K$ is a $K$-equivariant isomorphism.  An
{\em isomorphism} of framed vortices is an isomorphism of the
underlying marked vortices, intertwining the framings: if $\psi: P \to
P'$ denotes the bundle isomorphism, then $ \phi_j = \phi_j' \circ
\psi, j = 1,\ldots, n$.
\end{definition} 
\noindent Let $M_n^K(\Sigma,X)_\rho$ denote the moduli space of
isomorphism classes $n$-marked $\rho$-vortices.  The moduli space
$M_n^K(\Sigma,X)_\rho$ is homeomorphic to the product $
M^K(\Sigma,X)_\rho \times M_n(\Sigma)$
where $M_n(\Sigma)$ denotes the configuration space of $n$-tuples of
distinct points on $\Sigma$.  Let $M_n^{K,\fr}(\Sigma,X)_\rho$ denote
the moduli space of isomorphism classes of framed $n$-marked
$\rho$-vortices.  The {\em framed evaluation map} is
$$ \ev^\fr: M_n^{K,\fr}(\Sigma,X)_\rho \to X^n,
\ \ \ [A,u,\ul{z},\ul{\phi}] \mapsto (\phi_1(u(z_1)),\ldots,
\phi_n(u(z_n))) $$
defined by combining the framings with evaluation at the marked
points.  Define 
$$ \varphi: M_n^{K,\fr}(\Sigma,X)_\rho \to M_n^K(\Sigma,X)_\rho, \quad
[A,u,\ul{z},\ul{\phi}] \mapsto [A,u,\ul{z}] $$
by forgetting the framings.  Since the action of the gauge group
admits slices \eqref{slice} over the regular locus the map
$M_n^{K,\fr,\reg}(\Sigma,X)_\rho$ the map $\varphi$ is an orbifold
principal $K^n$-bundle.  Suppose that $K^n$ acts freely, so that
$M_n^{K,\fr,\reg}(\Sigma,X)_\rho$ is an honest bundle.  Let $\psi:
M^{K,\fr,\reg}_n(\Sigma,X)_\rho \to EK^n$ be a classifying map for
$\varphi$.  Combining $\psi$ with the framed evaluation map $\ev^\fr$
gives rise to a $K^n$-equivariant map $ \ev^\fr \times \psi:
M_n^{K,\fr,\reg}(\Sigma,X)_\rho \to X^n \times EK^n$.  Define the {\em
  evaluation map} $ \ev_\Z: \ M_n^{K,\reg}(\Sigma,X)_\rho\to (X^n
\times EK^n)/K^n = X_K^n $ by descending $\ev^\fr \times \psi$ to the
quotient.  Pull-back by $\ev_\Z$ induces a map in equivariant
cohomology with integral coefficients.  More generally, the
classifying map exists after passing to the classifying space of
$M_n^{K,\fr,\reg}(\Sigma,X)_\rho$, and we obtain a pull-back for
cohomology in rational coefficients
\begin{equation} \label{evals}
 \ev^* : H_K(X,\Q)^n \to H(M_n^{K,\reg}(\Sigma,X)_\rho,\Q) .\end{equation} 

\subsection{Compactification}

There is a natural compactification of the moduli space
of vortices obtained by allowing bubbling:

\begin{definition} \label{nodal}  
Let $\Sigma$ be a compact connected smooth complex curve and $X$ a
Hamiltonian $K$-manifold.  A {\em nodal gauged holomorphic map} to $X$
with principal component $\Sigma$ consists of a datum
$(P,A,C,{u},\ul{z})$ where $P$ is a principal $K$-bundle on $\Sigma$;
$A$ is a connection on $P$; $C$ is a compact nodal curve, $u: C \to
P(X)$ is a holomorphic map such that the composition $\pi_{P(X)} \circ
u: C \to \Sigma$ has homology class $[\Sigma]$; $\ul{z} =
(z_1,\ldots,z_n) \in {C}^n$ are distinct, smooth points of ${C}$.  The
{\em principal component} of a nodal gauged holomorphic map is the
(unmarked) gauged holomorphic map $(P,A,C_0,u_0)$ where $C_0 \subset
C$ is the unique component of $C$ on which the composition $\pi_{P(X)}
\circ u | C_0$ is non-constant.  A {\em nodal vortex} is a nodal
gauged map such that the principal component $(A,u_0)$ is a vortex.  A
nodal vortex is \emph{polystable} if the underlying map $u$ is
stable. A polystable vortex is {\em stable} if it has finite
automorphism group.  An {\em isomorphism} of nodal $\rho$-vortices
$(P,A,C,{u},\ul{z}),(P',A',C',{u}',\ul{z}')$ consists of an
isomorphism of nodal curves $f: {C} \to {C}'$ and a bundle isomorphism
$k: P \to P'$ such that $ kA = A'$, $ u' \circ f = k u$, and $f(z_i) =
z_i', i = 1,\ldots, n $.  A {\em framed nodal vortex} consists of a
nodal vortex together with framings at the attaching points of the
bubbles $ \phi_i : P_{\hat{z_i}} \to K, i = 1,\ldots, n .$ The {\em
  combinatorial type} of a gauged nodal map $(P,A,C,{u},\ul{z})$ is
the rooted graph $\Gamma$ whose vertices are the components of $C$,
whose finite edges represent the nodes $w_i^\pm, i =1,\ldots, m$,
semi-infinite edges represent the markings $z_1,\ldots, z_n$, and
whose root vertex represents the principal component $C_0$.  The {\em
  homology class} of a nodal vortex is the class $u_{K,*}[C] \in
H_2^K(X,\Z)$ where $u_K: C \to EK \times_K X$ is the lift of a
classifying map $C \to EK$ of $(\pi_{P(X)} \circ u)^* P \to C$
corresponding to $u$.
\end{definition} 

We often abbreviate the data of a nodal vortex as $(A,u)$ to save
space.  Let $M_{n,\Gamma}^K(\Sigma,X,d)_\rho$
resp. $M_{n,\Gamma}^{K,\fr}(\Sigma,X,d)_\rho$ be the moduli space of
isomorphism classes of polystable resp. framed polystable
$\rho$-vortices of combinatorial type $\Gamma$ and homology class $d
\in H_2^K(X,\Z)$, and $\ol{M}^K_n(\Sigma,X,d)_\rho$ be the union over
types
$$ \ol{M}^K_n(\Sigma,X,d)_\rho = \bigcup_\Gamma M_{n,\Gamma}^K(\Sigma,X,d)_\rho .$$
A notion of {\em Gromov convergence} of nodal vortices, described in
\cite{ott:remov}, defines a topology on $\ol{M}^K_n(\Sigma,X,d)_\rho$.

To define the linearized operator associated to a vortex, we recall
that the {\em normalization} $\ti{C}$ of a nodal curve $C$ is the
disjoint union of the irreducible components.  Let $\Omega^0(C,{u}^*
T^{\on{vert}} P(X))_{1,p} $ denote the subspace of
$\Omega^0(\ti{C},{u}^* T^{\on{vert}} P(X))_{1,p} $ consisting of
sections that agree at the nodes.

\begin{definition}  
Given a nodal vortex $(A \in \A(P),u: C \to P(X))$, let
$\ti{D}^\rho_{A,{u}}$ denote the {\em linearized operator}
\begin{multline}
  \Omega^1(\Sigma,P(\k))_{1,p} \oplus \Omega^0(C,{u}^* T^{\on{vert}}
  P(X))_{1,p} \\ \to (\Omega^{0} \oplus \Omega^2)(\Sigma,P(\k))_{0,p}
  \oplus \Omega^{0,1}(\ti{C},{u}^* T^{\on{vert}} P(X))_{0,p}
  \\(a,v) \mapsto (\ti{D}_{A,u_0}(a,v_0), (D_{u_j}
  v_j)_{j=1}^k )
\end{multline}
given by the operator $\ti{D}_{A,u}^\rho$ on the principal component
and the linearized Cauchy-Riemann operator $\ti{D}_{v_j}$ on the
bubbles.  We say that $(A,u)$ is {\em regular} iff
$\ti{D}^\rho_{A,{u}}$ is surjective.
\end{definition} 

\begin{theorem} \label{compact1} Let $X$ be a compact Hamiltonian
  $K$-manifold.  For any $c > 0$, the union of components
  $\ol{M}^K_n(\Sigma,X,d)_\rho$ with $ \lan d,[\omega_K] \ran < c$ is
  a compact, Hausdorff space.  The regular locus
  $\ol{M}^{K,\reg}(\Sigma,X,d)_\rho$ has the structure of a partially
  smooth topological (non-canonically $C^1$) orbifold.
\end{theorem} 
\noindent This generalization of Gromov compactness is proved in Ott
\cite{ott:remov}, with special cases proved previously in Mundet
\cite{mun:ham}.  That the topology defined by Gromov convergence is
compact follows using local distance functions as in
\cite[p. 134]{ms:jh}.  The orbifold charts are given by universal
deformations constructed in \cite{deform}.  The evaluation maps on the
principal stratum defined in \eqref{evals} extend to maps
$$ \ev^* : H_K(X,\Q)^n \to H(\ol{M}_n^{K,\reg}(\Sigma,X)_\rho,\Q) $$
and the forgetful morphism extends to map 
$f: \ol{M}^K_n(\Sigma,X)_\rho \to \ol{M}^K_n(\Sigma)$
forgetting $A$, replacing $u$ with its
composition with the projection $\pi_{P(X)}$, and collapsing any
unstable component.  

\subsection{Semistable gauged maps} 

Let $X$ be a smooth projectively-embedded $G$-variety.  Recall that a
{\em principal $G$-bundle} over a scheme $S$ is an
$S$-scheme $P$ with a right $G$-action such that $P$ is locally trivial in
the \'etale topology on $S$.

\begin{definition}  
 An $n$-marked {\em nodal gauged map} to $X$ over a scheme $S$ with
 principal component $\Sigma$ consists of a datum $(P,C,u,\ul{z})$
 where $(C,\ul{z})$ is a family of $n$-marked pre-stable curves (see
 e.g. Behrend-Manin \cite{bm:gw}) of genus that of $\Sigma$, $P \to
 \Sigma \times S$ is a principal $G$-bundle; and ${u}: {C} \to P(X) :=
 (P \times X)/G$ is a family of stable maps such that composition of
 ${u}$ with the projection $P(X) \to \Sigma$ has homology class
 $[\Sigma]$.
\end{definition} 
\noindent The restriction on the homology class and genus means that
the domain $C$ of a nodal gauged map over a point $S = \{ s \}$ has a
principal component $C_0$ isomorphic to $\Sigma$ under the composition
of $C \to P(X)$ with $P(X) \to \Sigma$ and a number of bubble
components $C_1,\ldots, C_m$ mapping to points in $\Sigma$.

Mundet's stability condition \cite{mund:corr}, which combines that of
Ramanathan \cite{ra:th} for principal bundles and Mumford, for
finite-dimensional actions, is given as follows.  Recall that a
subgroup $R \subset G$ is {\em parabolic} iff $G/R$ is compact.  A
{\em Levi subgroup} of a parabolic subgroup $R$ is a maximal reductive
subgroup.  Each Levi subgroup has a {\em maximal unipotent}
complementary subgroup $U \subset G$ such that $R = LU$.  Quotienting
by $U$ defines a surjective homomorphism $p: R \to L$.  A {\em
  parabolic reduction} of a principal $G$-bundle $P$ consists of a
parabolic subgroup $R \subset G$, a principal $R$-bundle $P'$ and an
isomorphism $P'(G) \to P$.  There is a canonical bijection between
parabolic reductions with subgroup $R$ and sections $\sigma: \Sigma
\to P/R$, given by $\sigma \mapsto \sigma^*P$, the pull-back of the
$R$-bundle $P \to P/R$ under $\sigma: \Sigma \to P/R$.  Given such a
section $\sigma$ let $p_* \sigma^* P $ denote the associated
$L$-bundle.  Let $i: L \to G$ be the inclusion.

\begin{definition} 
Let $P \to \Sigma$ be a principal $G$-bundle and $u: \Sigma \to P(X)$
a section.  Given a parabolic reduction $\sigma: \Sigma \to P/R$, the
{\em associated graded bundle} is $\Gr(P) = i_* p_* \sigma^* P$; note
that $\Gr(P)$ has a canonical reduction to $L$ given by $p_* \sigma^*
P$.  Let $Z$ denote the center of $L$, $\z$ its Lie algebra, and
$\lambda \in \z$ antidominant such that $\exp(\lambda) = 1$.  The
family of automorphisms of $R$ given by conjugation by $z^{\lambda/2
  \pi i} = \exp(\ln(z) \lambda/ 2 \pi i)$ induces a family of bundles
$ P_{\sigma,\lambda} := (\sigma^* P \times \C) \times_R R$ (the action
of $R$ on itself is twisted by conjugation) with central fiber $(p_*
\sigma^* P)(R)$.  The stable section ${u}$ extends canonically over
the generic fibers by multiplication by $z^{\lambda/2 \pi i}$ and over
the central fiber, by properness of stable maps, to an {\em associated
  graded stable map} denoted $\Gr({u})$ from a nodal curve $C$ to
$(\Gr(P))(X)$.
\label{mundetsemistable}
Suppose that $\lambda$ defines a weight of $L$ via the inner product
on $\l \subset \g$.  The {\em degree} of the pair $(\sigma,\lambda)$
is defined in terms of the associated graded by
\begin{equation} \label{degree}
 \mdeg_{\sigma,\lambda}(P,u) = \int_\Sigma c_1( p_* \sigma^* P
 \times_L \C_\lambda) + \rho  \mu_\lambda( \Gr(u)_0) \Vol_\Sigma 
\end{equation}  
where $\Gr({u})_0: \Sigma \to P(X)$ denotes the principal component of
$\Gr({u}): {C} \to P(X)$, and $\mu_\lambda(\Gr(u)_0)$ is the weight of
the one-parameter subgroup $\C^*_\lambda$ generated by $\lambda$ on
any fiber $\mO_X(1)_{\Gr(u)_0(z)}, z \in \Sigma$.  The pair
$(\sigma,\lambda)$ is {\em destabilizing} iff $
\mdeg_{\sigma,\lambda}(P,u) > 0 $.  $(P,{u})$ is {\em unstable} if
there exists a de-stabilizing pair $(\sigma,\lambda)$, {\em
  semistable} if it is not unstable, {\em stable} if there are no
pairs with $\mdeg_{\sigma,\lambda}(P,u) \ge 0$, and {\em polystable}
if it is semistable but not stable and $(P,{u})$ is isomorphic to its
associated graded for any pair $(\sigma,\lambda)$ satisfying the above
with equality.
\end{definition}  

The weight $\mu_\lambda( \Gr(u)_0)$ can be described in terms of the
moment map as follows.  The underlying smooth principal $G$-bundle
admits a reduction of structure group from $G$ to $K$.  Suppose that
$X$ is equipped with moment map $\Phi: X \to \k^\dual \cong \k$
induced by the Fubini-Study equivariant symplectic form.  The map
$\Phi$ induces a map $P(\Phi): P(X) \to P(\k)$ taking values in
$P(\l)$ on the $\lambda$-fixed locus.  Being $L$-invariant $\lambda$
induces an element $P(\lambda) \in P(\l)$.  Then
\begin{equation}\label{eq:deg-mm} 
\mu_\lambda(\Gr(u)_0) = \lan P(\Phi) \circ \Gr({u})_0(z), P(\lambda)
\ran , \forall z \in \Sigma .
\end{equation}

Let $\ol{\frak{M}}^G_n(\Sigma,X)$ denote the category of nodal
$n$-marked gauged maps to $X$ with principal component $\Sigma$.

\begin{theorem}  \label{artin} \cite{cross} 
$\ol{\frak{M}}^G_n(\Sigma,X)$ has the structure of a (non-finite-type,
  non-separated) Artin stack.  The subcategory
  $\ol{\M}_n^G(\Sigma,X)_\rho$ of $\rho$-semistable gauged maps has
  the structure of an open substack.  If all automorphism groups are
  finite, $\ol{\M}_n^G(\Sigma,X)_\rho$ (or rather, any component
  $\ol{\M}_n^G(\Sigma,X,d)_\rho$ since $\ol{\M}_n^G(\Sigma,X)_\rho$ is
  not finite type) is a Deligne-Mumford stack equipped with a
  canonical perfect obstruction theory, and its coarse moduli space is
  homeomorphic to the moduli space of symplectic vortices
  $\ol{M}_n^K(\Sigma,X)_\rho$.  For each constant $c > 0$, the union
  of components $\ol{\M}_n^G(\Sigma,X,d)_\rho$ with homology class $d
  \in H_2^G(X)$ satisfying $\lan d, [\omega_G] \ran < c$ is proper.
\end{theorem} 

Let $\ol{\frak{M}}_n(\Sigma)$ be the moduli stack for $X$ and $G$
trivial; this is the moduli stack of $n$-marked pre-stable maps to
$\Sigma$ of class $[\Sigma]$ and genus that of $\Sigma$.  A relative
perfect obstruction theory for the morphism
$\ol{\M}_n^G(\Sigma,X,d)_\rho \to \ol{\frak{M}}_n(\Sigma)$ is defined
by
$R p_* (e^* T(X/G))^\dual$, where $p: \ol{\U}_n^G(\Sigma,X,d)_\rho \to
\ol{\M}_n^G(\Sigma,X,d)_\rho$ is the universal curve and $e:
\ol{\U}^G_n(\Sigma,X,d)_\rho \to X/G$ the universal morphism, together
with its canonical morphism to the cotangent complex.  Let $\ev:
\ol{\M}_n^G(\Sigma, X)_\rho \to (X/G)^n$ denote the evaluation map at
the marked points, $(P,C,u,\ul{z}) \mapsto (\ul{z}^* P,u \circ
\ul{z})$.  Let $\ev^*$ denote the induced pull-back in rational
cohomology of coarse moduli spaces,
$$ \ev^* : H_G(X,\Q)^n \to H(\ol{M}^G_n(\Sigma, X)_\rho,\Q) .$$
As explained in \cite{cross} there exists a forgetful morphism
$f:\ol{\M}_n^G(\Sigma,X)_\rho \to \ol{\M}_n(\Sigma) :=
\ol{\M}_{\on{genus}(\Sigma),n}(\Sigma,[\Sigma])$ which maps $(C,P,u,\ul{z})$ to the
stable map to $\Sigma$ obtained from $(C,\pi_{P(X)} \circ u,\ul{z})$
by collapsing unstable components.

\subsection{Gauged Gromov-Witten invariants} 

We assume that the reader is familiar with twisted Gromov-Witten
invariants as in Coates-Givental \cite{co:qrr}.  Given a
$G$-equivariant vector bundle $V \to X$, define the {\em index}
$$\Ind(V) = R p_* e^* V \in \Ob(D^b
\on{Coh}(\ol{\M}_n^G(\Sigma,X)_\rho)) $$  
where $p$, as before, is the projection from the universal curve and
$e$ is the universal morphism.  If $V$ is a $G$-representation, then
by $\Ind(V)$ we mean the index of the trivial vector bundle with fiber
$V$.  Denote the Euler class of $\Ind(V)$ by 
$$\Eul_{\C^*}(\Ind(V)) \in
H(\ol{M}^G_n(\Sigma,X)_\rho,\Q)[\zeta,\zeta^{-1}] ;$$
here $\zeta$ the equivariant parameter for the action of $\C^*$ by
scalar multiplication, acting trivially on the base.

\begin{definition}   Suppose that $X,G,\rho,d$ are as above 
so that every $\rho$-semistable gauged map of class $d$ has finite automorphism
group.  The {\em $V$-twisted gauged Gromov-Witten invariant} associated to
$\alpha \in H_G(X)^n, \beta \in H(\ol{M}_n(\Sigma))$ is
$$ \lan \alpha, \beta \ran_{\rho,d,V} :=
\int_{[\ol{\M}^G_n(\Sigma,X,d)_\rho]} \ev^* \alpha \cup f^* \beta
\cup \Eul_{\C^*} (\Ind(V)) \in \Q[\zeta,\zeta^{-1}] .$$
\end{definition} 

The splitting axiom for these invariants is discussed in \cite{cross}:
they form a {\em cohomological trace} on the cohomological field
theory defined by the equivariant Gromov-Witten invariants introduced
by Givental \cite{gi:eq}.  In the symplectic setting the invariants
can be defined under the assumption that every vortex is regular.
Pulling back and integrating using the usual orbifold fundamental
class gives a symplectic definition of the gauged Gromov-Witten
invariants.

\section{Vortices with zero area and genus}

 In this section we construct {\em zero-area gauged Gromov-Witten
   invariants} in genus zero.  There are two approaches: in the
 symplectic approach
the moduli space of zero-area vortices $\ol{M}_n^K(\P^1,X)_0$ is constructed via 
degenerate symplectic reduction on the moduli space of stable maps
$\ol{M}_{0,n}(\P^1 \times X).$
In the second approach, a 
moduli stack of zero-area gauged maps
$\ol{\M}_n^G(\P^1,X)_0$ 
is constructed via a stability condition
Definition \ref{d:zero-s} for $\ol{\M}_{0,n}(\P^1 \times X)$.

\subsection{Zero-area vortices}

Let $X$ be a 
compact Hamiltonian $K$-manifold with moment map $\Phi: X \to \k^\dual \cong \k$
equipped with an invariant compatible almost complex structure $J$.
Let $\Map(\P^1,X)$
denote the space of smooth maps from $\P^1$ to $X$.  Define
$$\phi: \Map(\P^1,X) \to \k, \ \ \phi(u) = \int_{\P^1} u^* \Phi
\Vol_{\P^1} .$$
\begin{definition}  A {\em zero-area-vortex} is a 
map $u \in \Map({\P^1},X)$ satisfying the equations $\phi(u) = 0$ and
$\olp u = 0$.  An {\em isomorphism} of zero-area vortices $u_0,u_1$ is
an element $k \in K$ with $ku_0 = u_1$.  A zero-area vortex is {\em
  stable} if it has finite automorphism group.
\end{definition}  
\noindent Let $M^K({\P^1},X,d)_0$ be the moduli space of
isomorphism classes of zero-area vortices of homology class $d \in H_2(X)$:
$$M^K({\P^1},X,d)_0 = \left\{ u \in \Map(\P^1,X) \ \left| \ \olp u =
0, \ \phi(u) = 0, \ u_* [\P^1] =d \right. \right\} / K .$$
Since any flat connection on $\P^1 \times K$ is gauge equivalent to
the trivial one and has automorphism group isomorphic to $K$, the
space $M^K({\P^1},X,d)$ may be viewed a moduli space of isomorphism
classes of pairs $(A,u)$ with $F_A = 0 , \olp_A u = 0$; this justifies
the use of the term vortex.  For $d \in H_2^K(X,\Z)$ we define
$M^K({\P^1},X,d)_0$ to be the union of components $M^K(\P^1,X,d')_0$
where $d' \in H_2(X,\Z)$ maps to $d$ under the inclusion of the fiber
$X \to EK \times_K X$.

The formal tangent space to $M^K({\P^1},X)_0$ at any point $[u]$ may
be identified with the kernel of a Fredholm operator, as follows.
Define
$$ E_u: \Omega^0({\P^1},u^*TX) \to \k, \ \ \ v \mapsto
\int_{\P^1} L_v \Phi \, \Vol_{\P^1}
$$
$$ E_u^*: \Omega^0({\P^1},u^* TX) \to \k, \ \ \ v \mapsto
\int_{\P^1} L_{Jv} \Phi\Vol_{\P^1} .$$
Thus $E_u(v)$ resp. $E_u^*(v)$ is the derivative of the average of
the moment map with respect to $v$ resp. $Jv$.  Then
$M^K({\P^1},X)_0$ is locally homeomorphic to the zero set of 
\begin{multline*} \cF_{u}^0 : \Omega^0({\P^1},u^* T^{\on{vert}} P(X))
  \to (\k\oplus\k) \oplus \Omega^{0,1}({\P^1},u^* T^{\on{vert}} P(X))
  \\ v \mapsto \left( \int_{\P^1} \exp_u(v)^*
      P(\Phi), E_u^* v, \Psi_u(v)^{-1} \ol{\partial}
      \exp_u(v) \right)
\end{multline*}
where $\Psi_u(v)^{-1}$ is parallel transport from $\exp_u(v)$ to $u$,
using a Hermitian connection.

\begin{definition} The {\em linearized operator} for a zero-area
  vortex $u$ is the linearization of $\cF_u^0$,
\begin{equation} \label{linzer} \ti{D}^0_{u}: \Omega^0({\P^1},u^*
    TX) \to (\k \oplus \k) \oplus \Omega^{0,1}({\P^1},u^* TX), \ \ \ \
    v \mapsto (E_u v, E_u^* v, D_{u} v ) .
\end{equation}
where $D_u$ is the standard linearized Cauchy-Riemann operator.  We
say that a zero-area-vortex $u$ is {\em regular} iff the operator
$\ti{D}^0_{u}$ is surjective.  The space of {\em infinitesimal
deformations} of $u$ is $ \Def(u) = \ker(\ti{D}^0_u) .$
\end{definition} 

\noindent Let $M^{K,\reg}({\P^1},X)_0 \subset M^K({\P^1},X)_0$ denote
the locus of regular zero-area vortices.  It has the structure of a
smooth, finite-dimensional orbifold; the proof is similar to that for
pseudoholomorphic maps and will be omitted.

\subsection{Compactification}

The moduli space $M^K(\P^1,X,d)_0$ admits a compactification by allowing
bubbling in the fibers, as in the case of finite vortex parameter
discussed in Theorem \ref{compact1}.  One way of constructing this
compactification is by the following quotient construction as a formal
symplectic quotient of the moduli space of stable maps.  Let
$\ol{M}_{0,n}(\P^1 \times X,(1,d))$ denote the moduli space of
parametrized genus $0$ stable holomorphic maps ${u} = ({u}_1,{u}_2) :
{C} \to \P^1 \times X$ of homology class $(1,d)$.  Define
$$ \phi: \ol{M}_{0,n}(\P^1 \times X,(1,d)) \to \k, \ \ \ {u}
\mapsto \int_{{C}} {u}_2^* \Phi \  {u}_1^* \Vol_{\P^1} .$$
The map $\phi$ can be considered as a moment map for the action of $K$
on $\ol{M}_{0,n}(\P^1 \times X,(1,d))$ for a closed two-form given by
integrating the pull-back of $\omega \in \Omega^2(X)$ over the
principal component as in \eqref{sympform}.  By definition of
zero-area vortices there is an isomorphism
\begin{equation} \label{quotient}
 \ol{M}_n^K(\P^1,X,d)_0 \cong \phi^{-1}(0)/K =: \ol{M}_{0,n}(\P^1 \times
 X,(1,d)) \qu K .\end{equation}

The local structure of the regular locus is described as follows:

\begin{definition}  Let $u:C\to \P^1 \times X$ be a zero-area vortex with combinatorial type
  $\Gamma$. The {\em linearized operator} for $u$ is
$$ \ti{D}^0_{{u}} : \Omega^0(C,{u}^*T(\P^1 \times X)) \to (\k \oplus
  \k) \oplus \Omega^{0,1}(\ti{C}, {u}^* T(\P^1 \times X)), \quad v
  \mapsto (E_u v, E_u^* v, D_u v) $$
where $D_u$ is the usual linearized Cauchy-Riemann operator on
$\Omega^0(C,u^* TX)$, $E_u, E_u^* $ are the operators on the principal
component and $\ti{C}$ is the normalization of $C$.  We say that ${u}$
is {\em regular} iff $\ti{D}^0_{{u}}$ is surjective.  The space of {\em
  infinitesimal deformations of constant type} of ${u}$ is
$$ \Def_\Gamma({u}) = \ker(\ti{D}^0_{{u}})/\aut({C}) .$$
The space of {\em infinitesimal deformations} of ${u}$ is the space of
infinitesimal deformations of fixed type, plus the space of gluing
parameters:
\begin{equation} \label{defs}  
\Def({u}) = \Def_\Gamma({u}) \oplus \bigoplus_{i=1}^m T_{w_i^+}
{C} \otimes T_{w_i^-} {C} \end{equation}
\end{definition} 
\noindent 
where $m$ is the number of nodes of $C$.

\begin{theorem}  \label{compact2}
For any $c> 0$, the union over $d \in H_2^K(X,\Z)$ with $\lan [\omega_K]
, d \ran < c$ of $\ol{M}_n^K({\P^1},X,d)_0$ is a compact, Hausdorff
space.  The regular locus $\ol{M}_n^{K,\reg}({\P^1},X)_0$
admits the structure of a stratified-smooth topological orbifold, with tangent
space at $[u]$ isomorphic to $\Def(u)$. 
\end{theorem} 

\begin{proof}  This follows from the symplectic quotient description 
\eqref{quotient} and compactness and regularity properties of the
moduli space of stable maps.
\end{proof} 

\subsection{Zero-semistability}

In this section we give an algebraic interpretation of the moduli
space of zero-area vortices in terms of a stability condition.  Let
$X$ be a smooth projectively-embedded $G$-variety.
The $G$-action on $X$ induces an action
on the stack $\ol{\M}_{0,n}(\P^1 \times X,(1,d))$ by translation, see
Romagny \cite{ro:gr} for discussion of group actions on stacks.

\begin{definition}\label{d:zero-s}  Given a non-zero element $\lambda \in \k$
  and a stable map ${u}: {C} \to {\P^1} \times X$, the {\em associated
    graded} map $\Gr(u)$ is the Gromov limit of maps $\exp(t i
  \lambda) {u}$ as $t \to - \infty$.
 \label{0ssdef} The {\em degree} of $\lambda$ with respect
to ${u}$ is $\mdeg_\lambda({u}) := \lan \Gr({u})_0^* \Phi,
\lambda \ran $, 
where the pairing is taken at any point in the principal component.  A
stable map ${u}: {C} \to {\P^1} \times X$ is {\em $0$-semistable} iff
$\mdeg_\lambda({u}) \leq 0 $ for all $\lambda \in \g$, {\em
  $0$-unstable} iff it is not $0$-semistable, and $0$-stable iff it is
$0$-semistable with strict inequality for each non-zero $\lambda$.  An {\em
  isomorphism} of $0$-semistable maps $u: C \to {\P^1} \times X, u':
C' \to {\P^1} \times X$ is an isomorphism $\psi: C \to C'$ and an element $g
\in G$  such that $gu = u' \circ \psi$.  
\end{definition} 

A stable map is $0$-unstable iff a generic value on the principal
component is $0$-unstable for some one-parameter subgroup, which is a
closed condition.  Hence the $0$-semistable locus is open in
$\ol{\M}_{0,n}({\P^1} \times X)$.  (There is a seemingly unavoidable
conflict between stability terminology here: a stable map may or may
not be $0$-(semi)stable.)  We denote by $\ol{\M}_n^G({\P^1},X)_0$ the
stack of $0$-semistable maps to $X$.  More explicitly
$\ol{\M}_n^G({\P^1},X)_0$ is a substack of the quotient stack
$\ol{\M}_{0,n}({\P^1} \times X)/G$ of $\ol{\M}_{0,n}({\P^1} \times X)$
by $G$.  An object of $\ol{\M}_{0,n}({\P^1} \times X)/G$ over a scheme
$S$ consists a principal $G$-bundle $P \to S$ and a stable map $u: C
\to \P^1 \times P(X)$ over $S$.  Artin charts for
$\ol{\M}_{0,n}({\P^1} \times X)/G$ are induced by charts for
$\ol{\M}_{0,n}({\P^1} \times X)$, making $ \ol{\M}_n^G({\P^1},X)_0
\subset \ol{\M}_{0,n}({\P^1} \times X)/G$ into an Artin stack
cf. \cite{ro:gr}, \cite[Appendix]{jo:rr1}.

\begin{theorem} \label{proper}  
Suppose that every $0$-semistable gauged map is stable.  For any
constant $c > 0$, the union of components $\ol{\M}_n^G({\P^1},X,d)_0$
with homology class $d \in H_2(X,\Z)$ satisfying $\lan d , [\omega]
\ran < c$ is a proper Deligne-Mumford stack with a relative perfect
obstruction theory over $\ol{\frak{M}}_n(\P^1)$. 
\end{theorem} 

To show properness of $\ol{\M}_n^G(\P^1,X)_0$ we use the following
result, similar to Kirwan \cite{ki:coh} and Ness \cite{ne:st}, which
gives an equivalence between the algebro-geometric and
symplecto-geometric definitions.  We say that $\lambda \in \k$ is a
{\em maximally destabilizing} vector iff for any $\xi \in \k$,
$\mdeg_\xi(u) / \Vert \xi \Vert \leq \mdeg_\lambda(u) / \Vert \lambda
\Vert$ with equality iff $\R_{> 0} \xi = \R_{> 0} \lambda$.

\begin{proposition}  \label{typestrat}
Let ${u} : {C} \to {\P^1} \times X$ be a stable map.  Then
\begin{enumerate}
\item the map $u$ is $0$-semistable iff $\ol{Gu} \cap \phi^{-1}(0) \neq \emptyset$;
\item  $Gu \cap \phi^{-1}(0)$ contains at most one $K$-orbit;
\item Any $0$-unstable ${u}$ has a maximally destabilizing vector $\lambda$.
\end{enumerate}
\end{proposition} 

\begin{proof} In Kempf-Ness \cite{ke:le} and Ness \cite{ne:st}  
the case of projective or K\"ahler actions was considered.  The case
  of degenerate polarizations, when the 2-form is only required to be
  non-degenerate on each orbit, is discussed in \cite[Remark
  7.2.5]{quotients}.
\end{proof}  

\begin{remark} 
The theory of the {\em Jordan-H\"older} vector is missing.  That is,
we do not know whether, if $\ol{Gu}$ contains a vortex, whether the
orbit of such a vortex is unique.  
\end{remark} 

\begin{proposition} Suppose that every 
$0$-area vortex has finite automorphism group. Then the coarse moduli
  space of $\ol{\M}_n^G(\P^1,X)_0$ is homeomorphic to the moduli space
  of $0$-area vortices $\ol{M}^K_n(\P^1,X)_0$.
  \end{proposition} 

\begin{proof} The assumption on finite automorphism group implies
stable=semistable, since the associated graded of any $0$-semistable
but not $0$-stable has infinite automorphism group.  By part (a) of
\ref{typestrat}, any $0$-stable map $u$ is complex orbit equivalent to
a $0$-area vortex; the resulting map from the coarse moduli space (the
quotient of the semistable locus by the action) has a continuous
inverse given by the inclusion of $\phinv(0)$ into the $0$-stable
locus.
\end{proof}

\begin{proof}[Proof of Theorem \ref{proper}]   
Properness follows from properness of the coarse moduli space in
Theorem \ref{compact2}; the case of non-zero vortex parameter is
discussed in \cite{cross}.  A relative perfect obstruction theory is
given by $Rp_* (e^* T(X/G))^\dual$ as in Behrend \cite{be:gw} and the
canonical map to the cotangent complex, where $p$ is the universal
curve and $e$ the universal morphism, c.f. Gonzalez-Woodward \cite[Theorem 4.23]{cross}.
\end{proof}

\subsection{Zero-area gauged Gromov-Witten invariants}  

Suppose that $X$ is a smooth projectively-embedded $G$-variety.  Given $u: C \to
{\P^1} \times X$ we denote by $u_1: C \to {\P^1}$ the component with
values in ${\P^1}$.

\begin{lemma} There exists a forgetful morphism 
$f:\ol{\M}_n^G({\P^1},X)_0 \to \ol{\M}_n({\P^1})$ which maps
  $(P,C,u,\ul{z})$ to the stable map obtained from $(C,u_1,\ul{z})$ by
  collapsing unstable components.
\end{lemma} 

\begin{proof}  The existence of a morphism $\ol{\M}_{0,n}({\P^1} \times X)
\to \ol{\M}_{0,n}({\P^1})$ 
collapsing unstable components follows from
the functoriality properties of the moduli space of stable maps given
in Behrend-Manin \cite{bm:gw}.  The restriction to the $0$-semistable
locus is invariant and so factors through the quotient
$\ol{\M}_n^G({\P^1}, X)_0$.
\end{proof}

Let $\ev: \ol{\M}_n^G({\P^1}, X)_0 \to (X/G)^n$ denote the evaluation
map at the marked points.  Let $\ev^*$ denote the induced pull-back in
rational cohomology of coarse moduli spaces $ \ev^* :  H_G(X,\Q)^n \to
H(\ol{M}^G_n({\P^1}, X)_0,\Q)  .$ 

\begin{definition} Suppose that every semistable zero-area 
gauged map of class $d \in H_2(X,\Z)$ has finite automorphism group.
The {\em zero-area gauged Gromov-Witten invariant} for $\alpha \in
H_G(X,\Q)^n, \beta \in H(\ol{M}_n({\P^1}),\Q), d \in
H_2^G(X,\Z)$ is
\begin{equation} \label{vorinvinft} \lan \alpha ; \beta \ran_{d,0}
 = \int_{[\ol{\M}_n^G({\P^1},X,d)_0]} \ev^* \alpha \cup f^* \beta \in \Q
 .\end{equation}
\end{definition}

From the symplectic point of view, the invariants are defined without
virtual fundamental cycles if every vortex is regular, by integration
using the usual orbifold fundamental class.  At least in principle,
the zero-area gauged Gromov-Witten invariants may be computed from the
usual equivariant Gromov-Witten invariants using localization theorems
in \cite{je:lo1}, \cite{wo:norm}.

The splitting axiom for the zero-area gauged Gromov-Witten invariants
is the same as that for finite vortex parameter discussed in
\cite{cross}, that is, they form a cohomological trace on the
cohomological field theory defined by the equivariant Gromov-Witten
invariants introduced by Givental in \cite{gi:eq}.  We will not prove
the splitting axiom here; it follows from the equality with the gauged
Gromov-Witten invariants for small area proved later.

\section{The zero-area limit}

In this section we give both an algebraic and symplectic study 
of the moduli spaces in the zero-area limit.

\subsection{Small-area limit of Mundet-stability}

\begin{theorem} \label{projsame} Let $X$ be a smooth
  projectively-embedded $G$-variety.  There exists a $\rho_0 > 0$ such
  that for $\rho < \rho_0$, any gauged map $(P,C,{u})$ of genus zero
  is $\rho$-semistable iff $P$ is trivializable and (after identifying
  $P \to \P^1 \times G$ so that $P(X) \cong \P^1 \times X$) the map
  $u$ is $0$-semistable.
\end{theorem} 

\begin{proof}
For the proof we use the definition of degree \eqref{eq:deg-mm} using
the moment map. 
Let $\rho_0 >0$ be such that
$ \sup_{x \in X} \rho_0 \Vert \Phi(x) \Vert \int_{\P^1} \Vol_{\P^1} 
< 1 .$
Since the first Chern number of any line bundle is integral, a pair
$(\sigma,\lambda)$ in \eqref{degree} violates semistability for $\rho
< \rho_0$ iff
$$ \mdeg_{\sigma,\lambda}(P,u) = \int_{\P^1} c_1( p_* \sigma^* P
\times_L \C_\lambda) + \rho \lan P(\Phi) \circ \Gr({u})_0, \lambda
\ran \Vol_{\P^1} > 0 $$
iff
$ \int_\sigma c_1(p_* \sigma^* P \times_L \C_\lambda) \ge 0 $
and if equality holds then
\begin{equation} \label{0ss}
 \lan P(\Phi) \circ \Gr({u})_0, \lambda \ran  > 0
 .\end{equation}
It follows that $(P,C,{u})$ is $\rho$-semistable iff $P$ is
semistable, hence trivial (since ${\P^1}$ has genus zero), and ${u}$
satisfies the $0$-semistability condition of Definition
\ref{d:zero-s}.
\end{proof}  

\begin{corollary}  There exists
a $\rho_0 > 0$ such that for $\rho < \rho_0$, there is an equivalence
from $\ol{\M}_n^G({\P^1},X)_0$ to $\ol{\M}_n^G({\P^1},X)_\rho$ of
Deligne-Mumford stacks equipped with relative perfect obstruction theories.
\end{corollary} 

\begin{proof}   Theorem \ref{projsame} gives an equivalence from $\ol{\M}_n^G({\P^1},X)_0$ to
$\ol{\M}^G_n({\P^1},X)_\rho$ for $\rho$ as in the Theorem.  The
  relative perfect obstruction theories are equivalent, by
  definition.
\end{proof}  

Theorem \ref{equal} of the Introduction follows. 

\subsection{Small-area limit of vortices}

In this section we consider the small-area limit in the case that the
target is a Hamiltonian $K$-manifold, in which we lack a holomorphic
description of the moduli space.   
First we prove 
a result in the case without bubbling:

\begin{theorem} \label{compactnobubbles}
Suppose that $(A_\nu,u_\nu)$ is a sequence of $\rho_\nu$ vortices of
constant homology class $d \in H_2^K(X)$, with $\rho_\nu \to 0$.  If $
c_\nu = \sup | d_{A_\nu} u_\nu |$ is bounded, then after passing to a
subsequence, there exists a sequence of gauge transformations $k_\nu
\in \K(P)$ and a zero-area-vortex $u_\infty$ such that $k_\nu A_\nu $
converges to a trivial connection $A_\infty$ and (using the
trivialization induced by $A_\infty$) $k_\nu u_\nu \to u_\infty$
uniformly in all derivatives.
\end{theorem} 

\noindent The proof relies on the following well-known lemma,
cf. \cite[Lemma 2.3.10]{do:fo}, which controls the norm of the connection
linearly in terms of the norm of the curvature in genus zero:

\begin{lemma} \label{genzero}  Let $p \ge 1$.   
There exist constants $\delta > 0, c>0$ such that if $A$ is a
connection on the trivial bundle over ${\P^1}$ in Coulomb gauge with
respect to the trivial connection, that is, $\d^* A = 0$, then
$$\Vert A \Vert_{1,p} < \delta \implies \Vert A \Vert_{1,p} \leq c
\Vert F_A \Vert_{0,p} .$$
\end{lemma}

\begin{proof}  Using the elliptic estimate for $\d + \d^*$ and 
$H^1({\P^1})= 0$ we obtain $ \Vert A \Vert_{1,p} \leq c_1 \Vert \d A
\Vert_{0,p}$. The Sobolev multiplication estimate for $ [A,A]/2$
gives 
$$ \Vert \d A \Vert_{0,p} \leq c_2 (\Vert F_A \Vert_{0,p} + \Vert
A \Vert_{1,p}^2) ,$$
and thus for some $\delta,c>0$ we have $ \Vert A
\Vert_{1,p} \leq c \Vert F_A \Vert_{0,p}$ if $\Vert A \Vert_{1,p} <
\delta$.
\end{proof} 

\begin{proof}[Proof of Theorem \ref{compactnobubbles}]  Since $X$ is
  compact, the curvature of $A_\nu$ 
goes to zero as $\nu \to 0$ in $0,p$ norm,
$$ \Vert F_{A_\nu} \Vert^p_{0,p} \leq \rho_\nu \sup_{x \in X} |
\Phi(x) |^p \int_{\P^1} \Vol_{\P^1} \to 0 $$
using the vortex equation in Definition \ref{vdef}.  By weak Uhlenbeck
compactness, see e.g. \cite{we:ule}, after gauge transformation and
passing to a subsequence $A_\nu$ converges weakly in $W^{1,p}$ to a
connection $A_\infty$ for some $p > 2$.  Then $F_{A_\nu}$ converges
weakly in $L^p$ to $F_{A_\infty}$.  Since $F_{A_\nu}$ converges to $0$
in $L^p$ the convergence is strong and $F_{A_\infty}$ vanishes.  After
$W^{2,p}$ gauge transformation, we may assume that $A_\infty$ is
smooth.  Parallel transport by $A_\infty$ defines a trivialization of
$P$, since ${\P^1}$ is simply-connected.  Then $u_\nu$ defines a map
to $X$ and $A_\nu$ resp. $F_{A_\nu}$ defines a one-form resp.
two-form with values in the Lie algebra, for which we use the same
notation.  After gauge transformation, we may assume $\d^* A_\nu = 0$.
By Lemma \ref{genzero} $ \Vert A_\nu \Vert_{1,p}< c \rho_\nu, $ in
particular, $A_\nu$ converges to $0$ in $C^0$ norm.  Integrating the
vortex equation \eqref{veq} gives
$$\int_{\P^1} u_\nu^* \Phi \Vol_{\P^1} = 
\rho_\nu^{-1} \int_{\P^1}
 \hh [A_\nu,A_\nu] \Vol_{\P^1}  \to
0 .$$
Since $\sup |d_{A_\nu} u_\nu|$ is bounded, $u_\nu$ converges in $C^0$
to a section $u_\infty$ with
$$\int_{\P^1}
u_\infty^* \Phi \Vol_{\P^1} =  \lim_{\nu \to \infty} \int_{\P^1} u_\nu^* \Phi \Vol_{\P^1} = 0 .$$  
One obtains uniform convergence for $A_\nu,u_\nu$ in all derivatives
and holomorphicity of $u_\infty$ by elliptic bootstrapping as in
\cite[Section 3]{ci:symvortex}.
\end{proof}  

\noindent More generally for polystable vortices we have a similar
result Theorem \ref{limthm}. 

\begin{proof}[Proof of Theorem \ref{limthm}]     
As in the proof of Theorem \ref{compactnobubbles}, after gauge
transformation and passing to a subsequence we may assume that $A_\nu$
converges in $C^0$ to a flat connection $A_\infty$, which defines a
trivialization of $P$.  On any set on which the first derivative is
bounded, the principal component of ${u}_\nu$ converges to a
holomorphic map to $X$ and $A_\nu$ converges to $A_\infty$ uniformly
in all derivatives as before.  Gromov convergence of $u_\nu$ to a
stable map $u_\infty$ follows by Ott's arguments in \cite{ott:remov},
using the fact that $\rho_\nu \to 0$ implies that the bounds
\cite[(2.25)]{ott:remov} are satisfied.  In particular, if $u_{\nu,0}$
resp $u_{\infty,0}$ denotes the principal component then $ u_{\nu,0}$
converges to $u_{\infty,0}$ uniformly in all derivatives on compact
subsets of the complement of the (finite) bubbling set $Z$.
Integrating the vortex equation gives
\begin{eqnarray*}
 \Vert 
\int_{{\P^1}} u_{\infty,0}^* \Phi\Vert &=& 
\lim_{\eps \to
   0} \lim_{\nu \to \infty} 
\Vert \int_{\P^1 - B_\eps(Z)} u_{\nu,0}^*
   \Phi \Vert 
\\ &\leq&  \lim_{\eps \to 0} \lim_{\nu \to \infty}
 \Vert  \int_{\P^1} u_{\nu,0}^* \Phi\Vert + 
\Vol(B_\eps(Z)) \sup( \Vert \Phi \Vert) 
\\ &\leq & \lim_{\nu \to \infty} c \rho_\nu 
= 0 \end{eqnarray*} 
where $B_\eps(Z)$ denotes the union of $\eps$-balls around the points
in $Z$.
\end{proof} 

\section{Application to abelianization} 

In this section we prove the abelianization Theorem \ref{qmartinthm}
for gauged Gromov-Witten invariants of zero area, hence for gauged
Gromov-Witten invariants of small area.  We first give a topological
proof, assuming that $X$ is convex (that is, every pseudoholomorphic
map is regular) in which case the proof is essentially the same as
Martin \cite{mar:sy}.  Suppose that $K$ is a compact connected Lie
group.  Let $T \subset K$ denote a maximal torus and $W = N(T)/T$ its
Weyl group.  Consider the map $ H^T(X,\Z) \to H^K(X,\Z) $ induced by
the inclusion $T \to K$.  We write $d_T \mapsto d_K$ if $d_T \in
H_2^T(X,\Z)$ maps to $d_K \in H_2^K(X,\Z)$ under the isomorphism
$H_2^T(X,\Z) \to H_2^K(X,\Z)$ induced by $X \times_T ET \to X \times_K
EK$.  We denote by $\lan \alpha, \beta \ran_{K,d_K,\rho}$ resp.  $\lan
\alpha, \beta \ran_{T,d_T,\rho}^{twist}$ the untwisted
resp. $\g/\t$-twisted vortex invariants for $K$ resp. $T$.

\begin{proof}[Proof of \eqref{qmartininfty} in the convex case]
Suppose that $d_K \in H_2^K(X,\Z)$ is the image of $d \in H_2(X,\Z)$,
and $d_T$ is the image of $d$ in $H_2^T(X,\Z)$.  Let $\phinv_T(0)$
resp. $\phinv_K(0)$ denote the subspace of stable maps of homology
class $d$ with average $T$ resp. $K$ moment map equal to zero.  The
projection
$ \phinv_K(0)/T \to \phinv_K(0)/K $
has fiber $K/T$ with Euler characteristic
$\chi(K/T) = \chi((K/T)^T) = \# W $. The inclusion 
$ \phinv_K(0)/T \to \phinv_T(0)/T $
has normal bundle 
$ \phinv_K(0) \times_T (\k/\t)^\dual $
with transverse section induced by $\phi_K$.  Since the connections
are trivial in this case, the index of $(\k/\t)^\dual$ is $(\k/\t)^\dual$
itself.  Hence
\begin{eqnarray*} 
\lan \alpha, \beta \ran_{K,d_K,0} 
&=& \int_{\phinv_K(0)/K} \ev^* \alpha \cup f^* \beta \\
&=& (\# W)^{-1} 
\int_{\phinv_K(0)/T} 
\ev^* \alpha \cup f^* \beta \cup \Eul( \Ind(\k/\t))
\\ 
&=& (\# W)^{-1} \int_{\phinv_T(0)/T} \ev^* \alpha
\cup f^* \beta \cup \Eul( \Ind((\k/\t) \oplus (\k/\t)^\dual)) \\ 
&=& (\# W)^{-1} \lan \alpha, \beta \ran_{T,d_T,0}^{twist} .\end{eqnarray*}
\end{proof} 

For the case when the target an arbitrary smooth projectively-embedded
$G$-variety $X$ we derive the abelianization formula
\eqref{qmartininfty} from a version for sheaf cohomology.  A similar
strategy is used for various virtual integration identities in Joshua
\cite{jo:rr3}.  Recall from Lee \cite{lee:qk1} that the {\em virtual
  structure sheaf} of a Deligne-Mumford stack $\S$ equipped with a
perfect obstruction theory $\phi: E \to L_\S$ is defined as follows.
Let $\mE_\S = h^1/h^0(E^\dual)$ denote the corresponding cone stack
and $C_\S$ the intrinsic normal cone of Behrend-Fantechi \cite{bf:in}.
Define
$$ \O_{\S}^{\vir}
 := \O_{\S} \bigotimes^L_{0^{-1}_{\mE_\S}(\mO_{\mE_\S}) }
0^{-1}_{\mE_\S}(\mO_{C_\S}) $$
where $0_{\mE_\S}: \S \to \mE_\S$ is the inclusion of the vertex.  The
{\em virtual push-forward} of an object $F$ of the derived category is
$$ Rf_*^{\vir} F := Rf_* \left( F 
\bigotimes^L_{\mO_{\S_0}}
 \mO_{\S_0}^{\vir} \right) .$$
If $\S_1$ is a point then we denote by $\chi^{\vir}(F)$ the Euler
characteristic of $Rf_*^{\vir}(F)$.

We have the following equivariant version of abelianization which is
fairly trivial.  Denote by $\ul{\Lambda}(\g/\t)$ the vector bundle
with fiber the exterior algebra $\Lambda(\g/\t)$ on $\g/\t$.

\begin{proposition} 
\label{trivabel}
 For any $G$-equivariant vector bundle $E$ on $\ol{\M}_{0,n}(\P^1
 \times X,(1,d))$,
$
  \chi^{\vir}(E)^G 
=
( \# W)^{-1} \chi^{\vir}(E \otimes
 \ul{\Lambda}(\g/\t)))^T 
.$
\end{proposition}  

\begin{proof}  For any finite-dimensional representation 
$V$ of $G$ the dimension $ \dim(V^G)$ of the space of invariants is $
  ( \# W)^{-1} \dim (V \otimes \Lambda(\g/\t))^T $ by the Weyl
  character formula.  Applying this to the Euler characteristic on the
  left gives the result.
\end{proof} 

To derive abelianization from this fact we need the following
properties of local cohomology, the second of which uses a bit of
derived algebraic geometry; they are versions of the Thom and Gysin
isomorphisms in K-theory.  Suppose that $f: \S_0 \to \S_1$ is a
locally closed embedding.  We denote by $\chi^{\vir}_{\S_0}(\S_1,F)$
the Euler characteristic of the local cohomology sheaves
$$ \cH_{\S_0}^{\vir}(\S_1,F) := \cH_{\S_0}(\S_1,F \otimes^L
\mO_{\S_0}^{\vir} ) .$$
By Sch\"urg \cite{schurg:der}, any perfect
obstruction theory arises from a quasi-smooth derived structure.  A
derived structure on the moduli stack of stable maps, which is the
only case we use here, is discussed in \cite{toen:der}.

\begin{proposition}   Suppose that $F
  \to \S_0$ is a vector bundle.
\label{axioms}
\begin{enumerate}
\item Suppose that $f: \S_0 \to \S_1$ is a vector
  bundle.  Then $Rf_*^{\vir} F$ is isomorphic to $\iota^* F \otimes
  \Sym(\S_1^\dual)$ where $\iota: \S_1 \to \S_0$ is the zero section.
\item Let $f: \S_0 \to \S_1$ be a derived regular embedding with
  normal bundle $\cN$.  Then the local cohomology sheaf
  $\cH_{\S_0}^{\vir}(\S_1,F)$ admits a filtration by order of
  vanishing at $\S_0$ with associated graded complex $(\Sym(\cN)
  \otimes \det(\cN) \otimes f^* F \otimes^L \mO_{\S_1}^{\vir} )[
    \rank(\cN)]$.
\end{enumerate} 
\end{proposition} 

The Gysin isomorphism is (for schemes) Grothendieck \cite[p.14]{sga2}.
Local cohomology for derived algebraic geometry is developed in Lurie
\cite{lurie:dag12}, although we are certainly not using the full
theory here.

Local cohomology groups appear in the spectral sequence for the
Kirwan-Ness stratification for the action of $G$ on $\ol{\M} :=
\ol{\M}_{0,n}({\P^1} \times X, (1,d))$.  Recall first the Kirwan-Ness
stratification for the action of $G$ on $X$: Any non-zero $\lambda \in
\k$ defines a parabolic subgroup $R_\lambda$, the set of elements $r
\in R$ such that $\Ad(\exp(z i \lambda))r$ has a limit as $z \to
-\infty$.  The centralizer $G_\lambda$ of $\lambda$ is a Levi subgroup
and the map $r \to \lim_{z \to 0} \Ad(\exp(z \lambda))r$ is the
quotient by a maximal unipotent, and so a group homomorphism.  Let $X
= \bigcup_\lambda X_\lambda$ denote the Kirwan-Ness stratification of
$X$; here $\lambda$ ranges over equivalence classes of one-parameter
subgroups.  For each $\lambda$, let $Z_\lambda$ denote the
corresponding fixed point set of the one-parameter subgroup generated
by $\lambda$.  Let $Z_\lambda^{\ss}$ the semistable locus for the
action of $G_\lambda/\C^*_\lambda$ on $Z_\lambda$, $Y_\lambda$ the
subset of $X$ flowing to $Z_\lambda$ under $\exp(z \lambda)$, and
$Y_\lambda^{\ss}$ the inverse image of $Z_\lambda^{\ss}$.  Then (see
\cite{ki:coh}, \cite{ne:st})
$ X_\lambda = G \times_{R_\lambda} Y_\lambda^{\ss} .$

Strictly speaking, there is no Kirwan-Ness stratification for the
$G$-action on $\ol{\M}$ since there is no polarization.  However, the
stack $\ol{\M}$ has a similar stratification given as follows.

\begin{definition}   
Let $\M^\lambda $ denote the substack of $\ol{\M}$ consisting of
morphisms $u$ whose maximally destabilizing vector is conjugate to
$\lambda$.  Let ${\cZ}^\lambda$ denote the substack of maps that are
fixed by $\C^*_\lambda$ whose principal component takes values in
$Z_\lambda$; $\cZ^{\lambda,\ss}$ the locus in ${\cZ}^\lambda$ of
$0$-semistable maps for the residual action of
$G_\lambda/\C^*_\lambda$; ${\Y}^\lambda$ (resp. $\Y^{\lambda,\ss}$)
denote the substack of maps ${u}$ that flow under $\exp(z \lambda)$ to
${\cZ}^\lambda$ (resp. $\cZ^{\lambda,\ss}$).
\end{definition}

\begin{proposition} 
Each $\M^\lambda$ is a locally closed substack of $\ol{\M}$ equal to
$G \times_{G_\lambda} \Y^{\lambda,\ss} $.  The closure of $\M^\lambda$
is contained in the union of $\M^\nu$ with $\Vert \nu \Vert \geq \Vert
\lambda \Vert$.  The embedding $\M^\lambda \to \ol{\M}$ is the
truncation of a regular embedding of quasismooth derived stacks.
\end{proposition} 

\begin{proof}   That $\M^\lambda$ is locally closed
follows from the fact that the Hilbert-Mumford weight for $u$ with
respect to $\lambda$ is given by the weight for a generic value $u(z),
z \in C_0$.  The equality $\M^\lambda = G \times_{G_\lambda}
\Y^{\lambda,\ss} $ would follow from Kirwan \cite{ki:coh} applied to
the coarse moduli space $\ol{M}$ of $\ol{\M}$, except that $\ol{M}$ is
singular and has possibly degenerate two-form, where smooth.  However,
each $G$ orbit in $\ol{M}$ is smooth and has a non-degenerate K\"ahler
form, and \cite[Remark 7.2.5]{quotients} shows that the theory of
maximally destabilizing vectors extends to this degenerate case.  That
the closure of $\M^\lambda$ is contained in the union of $\M^\nu$ with
$\Vert \nu \Vert \leq \Vert \lambda \Vert$ follows from the
description of the norm $\Vert \lambda \Vert$ of the maximally
destabilizing vector $\lambda$ for $u$ as the infimum of $\Vert \phi
\Vert$ over the orbit $Gu$, see \cite[Lemma 5.4.3]{quotients}.
Regularity of $\M^\lambda \to \ol{\M}$ follows from embedding $X \to
\P^N$ which induces an embedding of $ \ol{\M}_{0,n}(\P^1 \times X)$
into the smooth stack $\ol{\M}_{0,n}(\P^1 \times \P^N)$, for which the
strata are honestly smooth.
\end{proof} 

The following is a version of non-abelian localization
for the action of $G$ on the stack $\ol{\M}$, compare \cite{te:qu}.

\begin{theorem} 
\label{nonab}  For $G$-equivariant vector bundle $V$ on $\ol{\M}$ 
we have 
\begin{multline} \chi^{\vir}(\ol{\M},V)^G = \sum_\lambda (-1)^{\codim({\M}^\lambda)}
\chi^{\vir}( {\cZ}^{\lambda,\ss}, V \otimes \\ \Sym(
T_{{\cZ}^{\lambda}} ({\cY}^\lambda)^\dual \oplus T_{{\M}^\lambda}
\ol{\M}) \otimes \det (T_{{\M}^\lambda} \ol{\M}) \otimes
\ul{\Lambda}(\g/\r_\lambda))^{G_\lambda}\end{multline}
where we have omitted
restrictions to simplify notation.
\end{theorem} 

\begin{proof} As in Teleman
  \cite{te:qu}.  By the spectral sequence associated to the stratification  
we have 
$$ \chi^{\vir}(\ol{\M},V)^G = \sum_\lambda
(-1)^{\codim(\M^\lambda)} 
\chi^{\vir}_{\M^\lambda} (\ol{\M},
V )^G .$$
Using the Gysin isomorphism \ref{axioms} (b),
$$ \chi^{\vir}_{\M^\lambda} ( \ol{\M},V )^G = 
(-1)^{\codim(\M^\lambda)}  \chi^{\vir}(\M^\lambda, V
\otimes \Sym( T_{\M^\lambda} (\ol{\M})) \otimes
\det (T_{\M^\lambda} \ol{\M}))^G .$$
Now the Thom isomorphism \ref{axioms} (a)
\footnote{Erratum added after publication: The proof of Theorem
  \ref{nonab} is not quite correct because for the derived structures
  on the strata as we have defined them the assumptions in
  \ref{axioms} (b) may not hold.  A more general non-abelian
  localization theorem, which applies to the case at hand by
  \cite[Appendix]{gr:loc}, has been proved by Daniel Halpern-Leistner
  \cite[(5)]{hl:qs}.}  and isomorphism $\M^\lambda \cong G
\times_{R_\lambda} \Y^\lambda$ imply
\begin{eqnarray*}  
&& \chi^{\vir}(\M^\lambda, V \otimes \Sym( T_{\M^\lambda}
  \ol{\M}) \otimes \det (T_{\M^\lambda} \ol{\M}))^G
  \\ &=& \chi^{\vir}(\Y^{\lambda,\ss}, V \otimes \Sym(
  T_{\M^\lambda} \ol{\M}) \otimes \det (T_{\M^\lambda}
  \ol{\M}) \otimes \ul{\Lambda}(\g/\r_\lambda))^{G_\lambda} \\ &=&
  \chi^{\vir}( \cZ^{\lambda,\ss}, V \otimes \Sym( T_{\M^\lambda}
  \ol{\M} \oplus T_{\cZ^\lambda}^\dual \Y^\lambda) \otimes
  \det (T_{\M^\lambda} \ol{\M}) \otimes
  \ul{\Lambda}(\g/\r_\lambda))^{G_\lambda} .
\end{eqnarray*}
\end{proof}  

We apply the formula in a situation where the only
contribution comes from the open stratum.  Let $\mO_X(1)$ be the
hyperplane bundle for $X$.  We construct a suitable determinant line
bundle on $\ol{\M}_n^G(\P^1,X)_0$ as follows.  Let $p:
\ol{\U}^G_n(\P^1,X)_0 \to \ol{\M}_n^G(\P^1,X)_0$ denote the universal
curve, and $q: \ol{\U}^G_n(\P^1,X)_0 \to \P^1$ the morphism given by
$(C,\pi:P \to \Sigma,u: C \to P(X),\ul{z},w \in C) \mapsto \pi(u(w))$.
Define
$$ D(k) = 
\det(Rp_* e^*\mO_X(1) \otimes 
q^* \mO_{\P^1}(k)) ;$$
the effect of the twist by $\mO_{\P^1}(k)$ is to make the contribution
from the principal component dominant.

\begin{lemma} \label{big} For any $\lambda$ and 
class $d \in H_2(X)$, there exists a constant $c$ such that the weight
of $\C^*_\lambda$ on $D(k) | {\cZ^\lambda}$ is at least $ k
(\lambda,\lambda) - c$.
\end{lemma}  

\begin{proof} We can compute the weight by Riemann-Roch
applied to any $\P^1$-fiber of $p$ over $\cZ^\lambda$.  For any
$\C^*_\lambda$-equivariant line bundle $L$ over $\ol{\U}^G_n(\P^1,X)_0 |
\cZ^\lambda$ we write the $\C^*_\lambda$-equivariant Chern class of
$L$ as a sum $c_1(L) + \mu_{L,\lambda}$, where $\mu_{L,\lambda}$ is
the weight on $L$.  The weight of $\C^*_\lambda$ acting on the
push-forward $\det(Rp_* e^*\mO_X(1) \otimes q^* \mO_{\P^1}(k))$ is the
equivariant part of the equivariant first Chern class of the
push-forward of the Chern character of $e^*\mO_X(1) \otimes q^*
\mO_{\P^1}(k)$.
The dominant contribution, linear in $k$, is 
$$ \left( \int_{\P^1} c_1(\mO_{\P^1}(k) ) \right)
\mu_{\mO_X(1),\lambda} = k (\lambda,\lambda) .$$
The actual weight is given by the above formula plus a zeroth order
term, as claimed. \end{proof}  

\noindent We denote by $\kappa_G(V)$ resp. $\kappa_T(V)$ the quotient
of a $G$-equivariant vector bundle $V$ on $\ol{\M}_{0,n}(\P^1 \times
X,(1,d))$ to $\ol{\M}^G_n(\P^1,X,(1,d))_{0}$ resp.
$\ol{\M}^T_n(\P^1,X,(1,d))_{0}$.

\begin{corollary} \label{pos}  For any $G$-vector bundle
$V$ over $\ol{\M}_{0,n}(\P^1 \times X,(1,d))$, for $k \gg 0$ we have
$$ \chi^{\vir}(\ol{\M}_n^G(\P^1,X,d)_0, \kappa_G(V \otimes D(k))) =
  \chi^{\vir}( \ol{\M}_{0,n}(\P^1 \times X,(1,d)), V \otimes D(k))^G .$$
\end{corollary}  

\begin{proof}  
By Theorem \ref{nonab} and Lemma \ref{big}, for $k \gg 0$ the only
contribution to the invariant part $
\chi^{\vir}(\ol{\M}_{0,n}(\P^1,X,d), \kappa_G(V))^G$ comes from the
$0$-semistable stratum.  The result now follows from Proposition
\ref{trivabel}.
\end{proof}

The final step is pass to cohomology using the virtual Riemann-Roch
formula proved in Tonita \cite{to:rr}.  To justify the application of
the virtual Riemann-Roch formula, we must show that
$\ol{\M}^G_n(\P^1,X,d)_{0}$ embeds in a non-singular proper
Deligne-Mumford stack.  For this note that $\ol{\M}_{0,n}(\P^1 \times
X,(1,d))$ embeds equivariantly in some $\ol{\M}_{0,n}(\P^N,\iota_*
(1,d))$ which is non-singular, where $\iota:\P^1 \times X \to \P^N$ is
a projective embedding, and so $\ol{\M}^G_n(\P^1,X,d)_0$ embeds in
$\ol{\M}_{0,n}(\P^N,\iota_* (1,d)) \qu G$.  Now as in Kirwan
\cite{ki:pa} we may assume, after blowing up recursively the
orbit-type strata, that $\ol{\M}_{0,n}(\P^N,\iota_* (1,d)) \qu G$ is a
locally free quotient and so a non-singular proper Deligne-Mumford
stack.

\begin{corollary} \label{descend}  Suppose that every 
$0$-semistable gauged map has finite automorphism group.  For any
  $G$-equivariant algebraic vector bundle $V$ on $\ol{\M}_{0,n}(\P^1
  \times X,(1,d))$,
$$ \chi^{\vir}(\ol{\M}_n^G(\P^1,X,d)_{0}, \kappa_G(V)) = (\#
    W)^{-1} 
\chi^{\vir}(\ol{\M}_n^T(\P^1,X,d)_{0}, \kappa_T(V
      \otimes \ul{\Lambda}(\g/\t))) .$$
\end{corollary} 

\begin{proof} 
By Corollary \ref{pos}, the statement of Corollary \ref{descend} holds
for bundles of the form $V \otimes D(k)$ for $k \gg 0$.  Virtual
Riemann-Roch implies that the virtual Euler characteristics
$\chi^{\vir}(\ol{\M}_n^G(\P^1,X,d)_0, \kappa_G(V \otimes D(k)))$ and $
\chi^{\vir}(\ol{\M}_n^T(\P^1,X,d)_{0}, \kappa_T(V \otimes D(k) \otimes
\ul{\Lambda}(\g/\t))) $ are quasipolynomial in $k$, and the claim
follows.  \end{proof}

Taking highest order terms on both sides of Corollary \ref{descend}
under the Adams operations, as in Teleman-Woodward \cite[Section
  5]{te:in}, gives
\begin{corollary} Suppose that any $0$-semistable gauged map has
finite automorphism group.  Then for any $d \in H^G_2(X,\Z)$ and
$G$-equivariant algebraic vector bundle $V$ on $\ol{\M}_{0,n}(\P^1 \times
X,(1,d)))$
$$ \int_{[\ol{\M}^G_n(\P^1,X,d)_{0}]} \kappa_G(\Ch_G(V)) = (\# W)^{-1}
\int_{[\ol{\M}^T_n(\P^1,X,d)_{0}]} \kappa_T (\Ch_G(V) \cup
\Eul(\Ind(\g/\t)) ) .$$
\end{corollary} 

Combining with the main result Theorem \ref{equal} gives the
abelianization Theorem \ref{qmartinthm} for genus zero gauged
Gromov-Witten invariants in the ``small area'' chamber, for classes of
the form $\alpha = \Ch_G(V)$ for algebraic vector bundles $V$.  By
Fulton-MacPherson \cite{fm:compact}, any class on $\ol{M}_n(\P^1)$ is
of this type.  By combining these results with those of \cite{cross},
one obtains the same result for arbitrary $\rho$ under suitable finite
automorphism assumptions.
 
\section{Construction of a cobordism}
 
One expects the equality of the zero-area gauged Gromov-Witten
invariants with the small-area gauged Gromov-Witten invariants in
Theorem \ref{equal} to hold more generally for any Hamiltonian
$K$-manifold.  We prove the equality in the case that every vortex is
regular, by constructing a cobordism between the two moduli spaces,
namely a differentiable structure on the union
$$\ol{M}_n^K({\P^1},X)_{[0,\rho]} := \bigcup_{\rho' \in [0,\rho]}
\ol{M}_n^K({\P^1},X)_{\rho'} .$$
Let $\ol{M}^K({\P^1},X)_{(0,\rho]}$ denote the locus with non-zero
  vortex parameter.  These spaces have natural topologies defined by
  Gromov convergence of the sections and $C^0$ convergence of the
  connections, up to gauge transformation.  The main result of this
  section is:

\begin{theorem} \label{cobord} Suppose that 
every element of $\ol{M}_n^K({\P^1},X)_0$ is regular. For $\rho$
sufficiently small, $\ol{M}_n^K({\P^1},X)_{[0,\rho]}$ has the structure
of an oriented stratified-smooth orbifold, giving an oriented
stratified-smooth orbifold cobordism between $\ol{M}_n^K({\P^1},X)_\rho$
and $\ol{M}_n^K({\P^1},X)_0$.
\end{theorem}

\subsection{Approximation for zero-area vortices} 

First we show that a regular zero-area vortex may be approximated by
sequences of small-area vortices. Let $A$ be the trivial connection on
$\Sigma \times K$.

\begin{theorem}\label{families} 
  Given a regular zero-area-vortex ${u} := (u_0,\ldots,u_m)$ of
  combinatorial type $\Gamma$ there exists a neighborhood $U$ of $0$
  in $\Def_\Gamma({u})$ and constants $c_0,\rho_0 > 0 $ such that for
  $\rho < \rho_0$ and $(a,v) \in U$ there exists a unique
  $(a_\rho,v_\rho)$ depending smoothly on
  $a,v$ such that
  $(A + \pi^*a_\rho, {u}_\rho = \exp_{{u}}(v + v_\rho))$
  is a $\rho$-vortex of combinatorial type $\Gamma$ in Coulomb gauge
  with respect to $(A,{u})$, $\Vert (a_\rho, v_\rho) \Vert < c_0
  \rho^{3/2}$, and $(a_\rho,v_\rho)$ in the image of the right
  inverse $Q^{0,\rho}$ of Lemma \ref{uniform}. 
\end{theorem}

The strategy of proof is to construct approximate solution and then
use an iteration to construct an exact solution.  It has in common
with the analogous theorem in Gaio-Salamon \cite{ga:gw} for the large
area limit that the quadratic term does {\em not} satisfy a uniform
bound. However, our case is easier because the space on which the
quadratic term is not uniformly bounded is finite dimensional.  The
approximation theorem implies that $\ol{M}^K({\P^1},X)_0$ is contained
in the closure of the union of moduli spaces $\ol{M}^K({\P^1},X)_\rho,
\rho \le \rho_0$, so that it provides a compactification in the usual
sense.  Given an zero-area-vortex $(A,u)$ with $A$ the trivial
connection we consider the equation
$$ F_{A + \pi^*a} + \rho \Vol_{\P^1} (\exp_u v)^* \Phi = 0 .$$
Consider the Hodge splitting
\begin{equation} \label{hodge} \Omega^2({\P^1},\k)_{1,p} = \Ker(\d^*)
  \oplus \Im(\d) \cong \k \oplus \Im(\d). \end{equation}
We denote by $\pi_0$ resp. $\pi_1$ the projection on the first
resp. second factor, 
$$ \pi_0(v) = (\int_{\P^1} v) \Vol_{\P^1}, \ \ \ \pi_1(v) = v
- \pi_0(v).
$$
Consider the following modification of \eqref{cutout} by using Banach
norms given by multiplying the standard $W^{1,p}$ resp. $L^p$, $p>2$
norms on the harmonic pieces by $\rho$.
\begin{equation} \label{cFr} 
\cF_{A,u}^{0,\rho}: \Omega^0({\P^1},\k)_{1,p} \oplus
\Omega^{0}({\P^1},u^*T X)_{1,p} \to \Im (\d_A \oplus \d_A*)_{0,p} \oplus
(\k \oplus \k) \oplus \Omega^{0,1}({\P^1},u^*T X)_{0,p}
\end{equation} 
$$ \cF_{A,u}^{0,\rho}:\bma a \\ v \ema \mapsto \bma \phi
\\ \psi \\ \mu \\ \lambda \\ \eta \ema := \bma \pi_1 (F_{A + \pi^*a} +
* \rho \exp_u(v)^* \Phi) \\ \pi_1(\d_A * a) \\ \pi_0
( * \exp_u(v)^* \Phi + \rho^{-1} ([a,a]/2)) \\ \pi_0 * L_{Jv} \Phi
\\ \Psi_u(v)^{-1} \ol{\partial}_{A + \pi^*a} \exp_u (v) \ema .$$
Here we have used that the term $\int_{\P^1} \d_A a$ vanishes by Stokes' theorem.  The map
$\cF_{A,u}^{0,\rho}$ of \eqref{cFr} is a smooth map of Banach spaces,
by standard Sobolev multiplication theorems.  The double subscript in
$\cF_{A,u}^{0,\rho}$ is meant to indicate that we view
$\cF_{A,u}^{\rho}$ as a map with respect to the Banach spaces for zero
area introduced above.

\begin{definition}  Let
$\ti{D}^{0,\rho}_{a,v}$ denote the linearization of
$\cF^{0,\rho}_{A,u}$ at $(a,v)$, that is,
\begin{equation} \label{inftyD} \ti{D}^{0,\rho}_{a,v}(a_1,v_1) =
  \bma \pi_1(\d_{A + \pi^*a} a_1 + \rho * L_{v_1} \Phi) \\
  \pi_1(\d_A * a_1) \\ \pi_0 ( * L_{v_1} \Phi + \rho^{-1} [a,a_1]/2 ) \\ \pi_0 * L_{Jv_1} \Phi \\ D_{A + \pi^*a ,\exp_u(v)}
  (a_1,v_1) \ema .
\end{equation} 
The operators $\ti{D}^{0,\rho}_{0,v}$ (that is, with $a
= 0$) have limit $\ti{D}^{0,0}_{0,v}$ given by 
$$
\ti{D}^{0,0}_{0,v}(a_1,v_1) = \bma \pi_1(\d_{A} a_1) \\
\pi_1(\d_A * a_1) \\ \pi_0 * L_{v_1} \Phi \\ \pi_0  * L_{Jv_1} \Phi \\
D_{A ,\exp_u(v)} (a_1,v_1) \ema .
$$
\end{definition} 
\noindent We wish to compare this operator with the linearized
operator $\ti{D}_u^0$ for the zero-area vortex $u$ of \eqref{linzer},
which does not have a gauge-theoretic part.  Let $A$ be the trivial
connection.  Note that a zero-area vortex $u$ is regular iff
$\ti{D}^{0,\rho}_{0,0}$ is surjective for $\rho$ sufficiently small.
Indeed, since the operator $\d_A \oplus \d_A^*$ for the trivial
connection $A$ is surjective onto the first component in the Hodge
decomposition \eqref{hodge}, $\ti{D}^{0,0}_{0,0}$ is surjective iff
$\ti{D}^0_{u}$ is.

Next we introduce suitable right inverses for the operators
$\ti{D}^{0,\rho}_{0,0}$.  

\begin{lemma} \label{uniform} If $(A,u)$ is a regular zero-area-vortex
  then the linearization $\ti{D}^{0,0}_{0,0}$ has a uniformly bounded
  right inverse $Q^{0,0}$ with the following property: There
  exists a constant $c > 0$ such that if $ Q^{0,0}
  (\phi,\psi,\mu,\lambda,\eta) = (a,v) $
then 
\begin{equation}
 \Vert a \Vert \leq c \Vert \phi,\psi \Vert, \quad \Vert v \Vert
 \leq c \Vert \phi,\psi,\mu,\lambda,{\eta} \Vert. \end{equation}
\end{lemma}

\begin{proof} 
Define $Q^{0,0}(\phi,\psi,\mu,\lambda,{\eta}) = (a,v) $ where $ a
= (\d_A \oplus \d_A^*)^{-1}(\phi,\psi) $, the right inverse to $\d_A
\oplus \d_A^*$ is given by the Hodge splitting, and
$$ 
(\pi_0 L_{v} \Phi, \pi_0 L_{J v} \Phi, D_{A,{u}}(0,v)) =  ( \mu,\lambda,{\eta} - D_{A,u} (a,0)) .$$
The claimed properties follow from standard Sobolev multiplication
theorems. 
\end{proof}

\begin{corollary}  
If $(A,u)$ is a regular zero-area-vortex then the linearization
$\ti{D}^{0,\rho}_{0,0}$ has a uniformly bounded right inverse
$Q^{0,\rho}$ with $Q^{0,\rho} = Q^{0,\rho}_{0} + \rho Q_1^{0,\rho} $
for some uniformly bounded operator $Q_1^{0,\rho}$, and $ Q^{0,\rho}_0
(\phi,\psi,\mu,\lambda,\eta) = (a,v) $ with $ \Vert a \Vert \leq c
\Vert \phi,\psi \Vert$ and $\Vert v \Vert \leq c \Vert
\phi,\psi,\mu,\lambda,{\eta} \Vert$.
\end{corollary}  

\begin{proof}  
Since $\ti{D}^{0,\rho}_{0,0} = \ti{D}_{0,0}^{0,0} + \rho
\ti{D}_{0,0}^{0,\rho,'}$ where $ \ti{D}_{0}^{0,\rho,'}(a_1,v_1) = (
\pi_1( * L_{v_1} \Phi), 0,0,0,0) $.
\end{proof}

\begin{lemma} \label{quadr}
The map $\cF^{0,\rho}_{A,u}$ satisfies a uniform quadratic
estimate except for a term quadratic in $a$ which has norm linear in
$\rho^{-1}$:
\begin{eqnarray*}
\cF_{A,u}^{0,\rho}(a + a_1, v + v_1) -
\cF_{A,u}^{0,\rho}(a,v)
&=& \ti{D}_{a,v}^{0,\rho}(a_1,v_1)
+ (\phi,0,\mu,0,\eta) 
\end{eqnarray*}
where 
%
$$ 
\Vert \phi \Vert \leq c_1 ( \Vert a_1 \Vert^2 + \rho \Vert v_1
\Vert^2) \quad \Vert \mu \Vert \leq c_2 ( \Vert v_1 \Vert^2 +
\rho^{-1} \Vert a_1 \Vert^2) \quad \Vert \eta \Vert \leq c_3 \Vert
a_1, v_1 \Vert^2
$$
for some constants $c_1,c_2,c_3$ depending on $\Vert A \Vert, \Vert u
\Vert$ and a bound on $\Vert a \Vert, \Vert v \Vert$.
\end{lemma}

\begin{proof} 
Define $\phi,\mu,\eta$ by
\begin{multline} 
\begin{array}{c}
\cF_{A,u}^{0,\rho}(a + a_1, v + v_1) \\ -
  \cF_{A,u}^{0,\rho}(a,v) \end{array}
= \bma \pi_1 \left( 
\begin{array}{ll}  \d_{A + \pi^*a}
  a_1 + [a_1,a_1]/2 + \rho \Vol_{\P^1} \\ ( * \exp_u(v + v_1)^* \Phi
  - \exp_u(v)^* \Phi) \end{array} \right) \\ \pi_1( \d_A * a_1)
\\ \pi_0 \left( \begin{array}{ll} * \exp_u(v + v_1)^* \Phi -
  \exp_u(v)^* \Phi \\ + \rho^{-1} (\d_{A + \pi^*a} a_1 +
      [a_1,a_1]/2) \end{array} \right) \\ \pi_0( * L_{Jv_1} \Phi )
\\ \Psi_{u}(v + v_1)^{-1} \ol{\partial}_{A + \pi^*(a + a_1)}
\exp_u (v + v_1) - \\ \Psi_u(v)^{-1} \ol{\partial}_{A + a}
\exp_u (v ) \ema \\ = \ti{D}^{0,\rho}_{a,v}(a_1,v_1) +
(\phi,0,\mu,0,\eta).
\end{multline}
The claimed estimates follow.  
\end{proof}

\begin{proof}[Proof of Theorem \ref{families}]
We give the proof in the case that $(A,u)$ has no bubbles; the general
case is similar and left to the reader.  We use Newton iteration to
find $a_\rho,v_\rho$ such that $\cF_{A,u}^{0,\rho}(a_\rho,v_\rho)
= 0 .$ To solve the Newton iteration we must show that our initial
condition is a sufficiently approximate solution so that the blow-up
of the quadratic term does not affect convergence of the iteration.
Let $Q^{0,\rho}$ denote the right inverse of Lemma \ref{uniform}.  We
define by induction a sequence
$$ \hat{\zeta}_\nu = - Q^{0,\rho} \cF^{0,\rho}_{A,u} \zeta_\nu,
\ \ \ \zeta_{\nu + 1} = \zeta_\nu + \hat{\zeta}_\nu $$
such that
\begin{equation} \label{hyp}
 \hat{\zeta}_\nu = (\hat{a}_\nu,\hat v_\nu), \ \ \ \Vert
\hat{a}_\nu \Vert \leq c_0  \rho^{\nu }, \ \ \ \Vert
\hat v_\nu \Vert \leq c_0 \rho^{\nu - 1} .
\end{equation} 
To get the iteration started we define
$$ \zeta_1 = \hat{\zeta}_0 = - Q^{0,\rho} (\rho * \exp_u(v)^*
\Phi,0,0,0,0) .$$
Hence 
$ \zeta_1 = (a_1,v_1)$ and $ \Vert \zeta_1 \Vert \leq c_4
\rho $
with $c_4$ depending on $\sup_{x \in X} | \Phi(x) |$ and the norm of
$Q^{0,\rho}$ which is uniformly bounded by Lemma \ref{uniform}.
Define
$$\hat{\zeta_1} = -Q^{0,\rho} \cF^{0,\rho}(\zeta_1), \ \ \
\zeta_2 = \zeta_1 + 
\hat{\zeta_1} = (a_2,v_2) .$$ 
Then
$$ \cF_{A,u}^{0,\rho}(\zeta_2) = (\phi_2,0,\mu_2,0,\eta_2) $$
and by Lemma \ref{quadr}
$$ \Vert \phi_2 \Vert, \Vert \eta_2 \Vert < c_5 \rho^2,
\ \ \ \ \Vert \mu_2 \Vert < c_6 \rho $$
for some constants $c_5, c_6$ depending on $c_1,c_2,c_3$.  Using Lemma
\ref{uniform} we have
$$\hat{\zeta_2} = (\hat{a}_2,\hat v_2), \ \ \ \Vert \hat{a}_2 \Vert
< c_7 \rho^2, \ \ \ \Vert \hat v_2 \Vert < c_8 \rho $$
for some constants $c_7,c_8$ depending on the previous constants and
the norm of $Q_1$.  

Suppose that the sequence $\hat{\zeta}_1,\ldots, \hat{\zeta}_{\nu-1}$
constructed in this way satisfies the hypotheses \eqref{hyp}.  By
Lemma \ref{quadr}, there exists a constant $c_9 > 0$ depending on the
previous constants such that
$$(\phi_\nu,\psi_\nu,\mu_\nu,\lambda_\nu,\eta_\nu) :=
\cF^{0,\rho}_{A,u}(\zeta_\nu)$$
satisfies for $\rho$ such that for $ \rho < 3 c_0 c_9 $
\begin{eqnarray*}
 \Vert \phi_\nu \Vert &<& c_9 ( \Vert a_\nu \Vert^2 + \rho \Vert v_\nu \Vert^2) 
< c_9 ( c_0^2 \rho^{2\nu} + c_0^2\rho^{2(\nu-1) + 1}) < c_0 
\rho^{ \nu + 1} \\
\Vert \mu_\nu \Vert &<& c_9  ( \Vert v_\nu \Vert^2 + \rho^{-1} \Vert a_\nu \Vert^2) <  c_9 c_0^2 ( \rho^{2(\nu-1)} + 
\rho^{2\nu-1})  <  c_0  \rho^\nu \\ 
\Vert \eta_\nu \Vert &<& c_9  \Vert a_1, v_1 \Vert^2 < c_9 c_0^2 ( \rho^{2 \nu} + \rho^{2 \nu - 1} + \rho)^{2 (\nu- 1) } < c_0 \rho^{ \nu}
  .\end{eqnarray*}
Applying $Q^{0,\rho}$ gives an element $\hat{\zeta}_\nu =
(\hat{a}_\nu,\hat v_\nu)$ satisfying \eqref{hyp}, as required.
Because $\Vert \hat{\zeta}_\nu \Vert < c_0 \rho^{\nu}$, $\zeta_\nu$ is
a Cauchy sequence and converges to a limit $\zeta_\rho =
(a_\rho,v_\rho)$ with
\begin{equation} \label{limit}
 \Vert \zeta_\rho - \zeta_\nu \Vert < c_0 \rho^{\nu + 1}/(1 - \rho) <
 c_0 \rho^{\nu + 1} .\end{equation}
To prove uniqueness, suppose that $\zeta_\rho' = (a_\rho',v_\rho')$
is another solution with
\begin{equation} \label{close} 
\Vert \zeta_\rho - \zeta'_\rho \Vert < c_0 \rho^{3/2}.
\end{equation} 
Then
$$
0 =  \cF^{0,\rho}_{A,u}(a_\rho,v_\rho) - 
\cF^{0,\rho}_{A,u}(a_\rho',v_\rho ' ) = D_{0}^{\rho,0} (a_\rho - a'_\rho,
v_\rho - v'_\rho) + \eps  $$
for some $\eps$. By the uniform quadratic estimate in Lemma
\ref{quadr} and \eqref{close}
\begin{equation} \label{small}
 \Vert \eps \Vert < c_0 \rho .\end{equation}
Since $ (a_\rho - a'_\rho ,v_\rho - v'_\rho)$ lies in the image of
$Q^{0,\rho}$, we have
\begin{equation} \label{lessthan} 
\Vert a_\rho - a'_\rho, v_\rho - v'_\rho \Vert
\geq c  \end{equation} 
which contradicts \eqref{small} for $c_0$ sufficiently small.
Smooth dependence follows from the implicit function theorem.
\end{proof} 

The approximation theorem \ref{families} extends to the case with
bubbles as follows.  For any collection of gluing parameters
$\ul{\delta} = (\delta_1,\ldots,\delta_m)$ and $v \in \Omega^0(C,{u}^*
TX)$, let
$\exp_{{u}^{\ul{\delta}}}( v^{\ul{\delta}}):
  {C}^{\ul{\delta}} \to X$
the glued zero-area vortex constructed using the implicit function
theorem as in \cite[Section 10]{ms:jh}, so that 
$$ \Def({u}) \to \ol{M}^K({\P^1},X)_0, \quad (v,\ul{\delta})
\mapsto \exp_{{u}^{\ul{\delta}}}( v^{\ul{\delta}}) $$
gives local charts for $\ol{M}^K({\P^1},X)_0$ as in \cite{deform}.  Thus
${C}^{\ul{\delta}}$ is obtained from ${C}$ by removing small balls
around the nodes, and gluing together the components using maps
$\kappa_j^+ = \kappa_j^-/\delta_j$, where $\kappa_j^\pm$ are fixed
local coordinates around the nodes, ${u}^{\ul{\delta}}$ is an
approximate solution constructed using cutoff functions, and
$v^{\ul{\delta}}$ is the correction provided by the implicit
function theorem.  Given a set of markings $\ul{z}$ on ${C}$ in the
complement of the domain of the local coordinates near the nodes we
denote by $\ul{z}^{\ul{\delta}}$ the markings on ${C}^{\ul{\delta}}$.  We
now wish to combine the gluing construction for pseudoholomorphic maps
with the approximation Theorem \ref{families}.  This will give rise to
charts for our moduli space near $\rho = 0$.

\begin{theorem}\label{families2} 
Given a regular zero-area-vortex ${u} := (u_0,\ldots,u_m)$
of combinatorial type $\Gamma$ there exists a neighborhood $U$ of $0$
in $\Def({u})$ and constants $c_0,\rho_0 > 0 $ such that for $\rho <
\rho_0$ and $(a,v,\ul{\delta}) \in U$ there exists a unique
$(a_\rho,v_\rho)$ such that
$$(A_{\on{triv}} + a + a_\rho, {u}_\rho =
\exp_{{u}^{\ul{\delta}}}(v^{\ul{\delta}} + v_\rho))$$
is a $\rho$-vortex in Coulomb gauge with respect to the trivial
connection $A_{\on{triv}}$, $\Vert a_\rho,v_\rho \Vert < c_0
\rho^{3/2}$, and $(a_\rho,v_\rho)$ in the image of the right
inverse $Q^{0,\rho}$ of Lemma \ref{uniform}.  
Furthermore,
$(a_\rho,v_\rho)$ depends stratified-smoothly on
$v,\ul{\delta}$, that is, smoothly on the subset where the
gluing parameters $\delta_i, i \in I$ are non-zero, for each $I
\subset \{1, \ldots, m \}$.
\end{theorem}

\begin{proof}  
Let $\exp_{{u}^{\ul{\delta}}}(v^{\ul{\delta}})$ denote the
zero-area vortex given by gluing.  We wish to solve
$ \cF_{A,\exp_{{u}^{\ul{\delta}}}(v^{\ul{\delta}})}^{0,\rho}(a +
a_\rho,v + v_\rho) = 0 $
for $a_\rho,v_\rho$, using Newton iteration. The error term
$\rho \exp_{{u}^{\ul{\delta}}}(v^{\ul{\delta}})^*
P(\Phi) $ is uniformly bounded in $\ul{\delta}$ and $\rho$, by
the bound on $\Phi$.  The norms of the operators $Q^{0,\rho}$
and $\ti{D}^{0,\rho}$ at
$\exp_{{u}^{\ul{\delta}}}(v^{\ul{\delta}})$ are uniformly
bounded in $\ul{\delta}$ and $\rho$, that is, 
$ \Vert Q^{0,\rho} \Vert < c $ and $\Vert \ti{D}^{0,\rho} \Vert < c $
for some $\rho$-independent constant $c$.  Indeed, the inverse to
$D_{\exp_{{u}^{\ul{\delta}}}(v^{\ul{\delta}})}$ is uniformly bounded
in $\ul{\delta}$, by the gluing argument for pseudoholomorphic maps in
\cite[Chapter 10]{ms:jh}.  The claim now follows from Lemma
\ref{uniform}.  The map $\cF^{0,\rho}_{A,{u},\ul{\delta}}$ satisfies a
quadratic estimate uniformly in $\ul{\delta}$ and uniformly in $\rho$
except for a term quadratic in $a$ which has norm linear in $\rho$,
just as in Lemma \ref{quadr}.  The same argument as in the case of
smooth domain in Theorem \ref{families} gives a solution
$a_\rho,v_\rho$ by Newton iteration.  Smoothness on each stratum
follows from smoothness of $v^{\ul{\delta}}$ on $v,\ul{\delta}$ and
the smoothness statement of Theorem \ref{families}.
\end{proof} 

\subsection{Surjectivity} 

Let ${u}$ be a stable zero-area vortex.  Let $\Def({u})_\eps$ denote a
$\eps$-ball around $0$ in the space $\Def({u})$ of infinitesimal
deformations of \eqref{defs}.  Theorem \ref{families} defines for
$\eps$ sufficiently small a map
\begin{equation} \label{Tu} T_{{u}}^\rho: \Def({u})_\eps \to
\ol{M}^K({\P^1},X)_\rho .\end{equation} 
The collection of images of the maps $T_{u}^\rho$ covers
$\ol{M}^K({\P^1},X)_\rho$ for $\rho$ sufficiently small:

\begin{theorem}   \label{surject} 
Suppose that every zero-area nodal vortex is regular.  For any
constant $c > 0$ and $d \in H_2^G(X,\Z)$, there exists $\rho_0 > 0 $
such that if $(A,{u})$ is a $\rho$-vortex with $\rho < \rho_0$ and
homology class $d$ then $(A,{u})$ is in the image of $T_{u'}^\rho$ for
some nodal zero-area vortex $u'$.
\end{theorem}  

\begin{proof}  
Suppose that the assertion in the statement does not hold, that is,
for every $\rho$ there exists a $\rho$-vortex $(A_\rho,{u}_\rho)$ of
class $d$ not in the image of any map $T_{u'}^\rho$. By the
compactness Theorem \ref{limthm}, after passing to a subsequence and
gauge transformations $(A_\rho,{u}_\rho)$ converges to a
zero-area-vortex ${u}_0$.  First consider the case that $u_0$ has
smooth domain $C \cong {\P^1}$, hence $u_\rho$ has smooth domain as
well.  We write
$$ A_\rho = a_\rho, \quad u_\rho = \exp_{u_0}(v_\rho) $$
for some $a_\rho, v_\rho$.  We may assume that $(A_\rho,u_\rho)$ is
in Coulomb gauge, that is,
\begin{equation} \label{Coulomb}
 \d^* a_\rho = 0 , \quad  E_{u_0}^* v_\rho = 0 .\end{equation} 
Define 
\begin{equation} \label{firstQ} (a_1,v_1) =  Q^{0,0}( \pi_1 \rho * u_\rho^* \Phi, 0,0,0,0 ) .\end{equation}
We have an estimate $\Vert a_1,v_1 \Vert < c_1 \rho$ where $c_1$
depends on the bound on $\Phi$ and the norm of $Q^{0,0}$.  Consider the
map
\begin{multline} \cF^{0,0}: \Omega^0({\P^1},\k)_{1,p} \oplus \Omega^{0}({C},(u_0)^*T
X)_{1,p} \to \Im (\d \oplus \d^*)_{0,p} \oplus (\k \oplus \k) \oplus
\Omega^{0,1}(C,(u_0)^*T X)_{0,p} \\
 ( a,v) \mapsto \left( \pi_1 F_a , \pi_1 d^* v , \int_{\P^1} *
\exp_{u_0}(v)^* \Phi , E_{u_0}^* v , \Psi_{u_0}(v)^{-1} \olp
\exp_{u_0}(v) \right) .\end{multline}
Using \eqref{Coulomb}, \eqref{firstQ}, we have
\begin{eqnarray*}
\Vert \cF^{0,0}(a_\rho -  a_1,v_\rho  - v_1) 
\Vert &=& 
\left\Vert
\left[ \begin{array}{c} \pi_1 F_{a_\rho - a_1} \\ 
\pi_1 d^* (a_\rho - a_1) \\
 \int_{\P^1} * \exp_{u_0}(v_\rho - v_1)^*\Phi  \\
 E_{u_0}^* v_\rho \\  
\Psi_{u_0}(v)^{-1} \olp \exp_{u_0}(v) - \\ \Psi_{u_0}(v -
v_\rho)^{-1} \olp \exp_{u_0}(v - v_\rho)  
\end{array} \right] \right\Vert
\\
&=& \left\Vert \left[ \begin{array}{c} 
\pi_1 (F_{a_\rho } - \d_A a_1 + 
\hh[ a_1,a_1]) \\  0 \\ 
\int_{\P^1} * \exp_{u_0}(v_\rho - v_1
)^* \Phi \\  0 \\ 
\Psi_{u}(v_\rho - v_1)^{-1} \olp \exp_{u_0}(v_\rho - v_1) - \\
\Psi_{u_0}(v_\rho)^{-1} \olp \exp_{u_0}(v_\rho) 
\end{array} \right] \right\Vert  \\
&\leq c_2& \rho^2  \end{eqnarray*} 
where $c_2$ depends on $c_1$ and the constants in the Sobolev
multiplication theorems.  By the implicit function theorem for the map
$\cF^{0,0}$, there exists a zero-area vortex $ {u}' =
\exp_{{u}_0}(v)$ within $c_3 \rho^2$ of $(A_\rho - a_1,
\exp_{u}( v_\rho - v_1)) $ in the sense that
$$ \Vert A_\rho - a_1 \Vert^2 + \Vert v - v_\rho + v_1 \Vert^2 <
c_3^2 \rho^4 .$$
Now consider the $\rho$-vortex $ (A'_\rho,u'_\rho) :=
T_{\rho,{u}_0}(v) =: (a'_\rho,\exp_{u'}(v'_\rho)) .$ As in the
proof of Theorem \ref{families2}, this vortex lies within $c_4 \rho^2$
of the first step in the Newton iteration, $ (A'_{\rho,1},u'_{\rho,1})
= (a_{\rho,1}, \exp_{u'}(v_{\rho,1})) $ where
\begin{equation} \label{secondQ} (a_{\rho,1},v_{\rho,1}) = Q^{0,0}( \pi_1 * \rho 
(u')^* \Phi, 0, 0, 0, 0  ) .\end{equation}
Now $(a_\rho,v_\rho)$ lies within $c_1 \rho^2$ of
$(a_{\rho,1},v_{\rho,1})$ by \eqref{firstQ} and hence within $c_5
\rho^2$ of $(a_\rho',v_\rho')$.  For $\rho$ sufficiently small,
$c_5 \rho^2 < c_0 \rho^{3/2}$.

Next we show that we may assume that $(a'_\rho,v'_\rho)$ lies in the
image of the right inverse $Q^{0,\rho}$ for the linearized operator
$\ti{D}^{0,\rho}$ of $u'$.  Given $(a,v)$ as the previous
paragraph, we claim that
\begin{equation} \label{form}
(a,v) = (a_0,v_0) + (a_1,v_1), \quad (a_0,v_0) \in \ker
  \ti{D}^{0,\rho}, \quad (a_1,v_1) \in \Im
  Q^{0,\rho} \end{equation} 
with norm $ \Vert (a_1,v_1) \Vert \leq c \Vert (a_1,v_1) \Vert.$
For any $c > 0$ there exists $\rho_0$ such that for $\rho < \rho_0$,
$ \Vert \ti{D}^{0,\rho} ( a_0,v_0)
\Vert \leq c \Vert (a_0,v_0) \Vert $
for any $(a_0,v_0) \in \ker \ti{D}^{0,0}$.  Thus the space
$\ker \ti{D}^{0,0}$ is transverse to the image of $Q^{0,\rho}$,
for $\rho$ sufficiently small: it meets $\Im Q^{0,\rho}$ trivially and
projects isomorphically onto $\ker \ti{D}^{0,\rho}$.  By the inverse
function theorem, any nearby pair $(a,v)$ is of the form
\eqref{form}. By the uniqueness statement in Theorem \ref{families2},
$(A_\rho,u_\rho)$ is in the image of the map $T_{u'}^{\rho}$.

The case with bubbles is similar, using a map obtained by combining
$\cF^{0,0}$ with the usual Cauchy-Riemann equation on the bubbles.  By
iteration one constructs a zero-area vortex $u'$ and a pair
$(a',v')$ such that $T_{u'}^\rho(a',v')$ is within $c_0
\rho^{3/2}$ of $(A,{u})$.  Surjectivity follows from the uniqueness
part of the implicit function theorem used in the gluing
construction.  \end{proof}

\subsection{Construction of charts}

The charts near zero area are given by the approximation construction
of the previous section.  Let $u$ be a zero-area vortex.  Using
\eqref{Tu} we have for $\rho$ sufficiently small a map
$$ T_{u} : \Def(u)_\meps \times [0,\rho] \to
\ol{M}^K({\P^1},X)_{[0,\rho]} .$$

\begin{proposition} \label{homeo}  For $\meps,\rho$ sufficiently small
$T_{u}$ is a homeomorphism onto an open neighborhood of $[u]$.
\end{proposition} 

\begin{proof}[Sketch of proof] 
By Theorem \ref{surject}, the image of $T_{u}$ contains an open
neighborhood of $u$.  We choose a complement of $\aut(C)$ in $\ker
\ti{D}_{u}$ and restrict $T_{u}$ to $v$ lying in this complement.
Suppose that $(A,u') = T_{u}(v,\ul{\delta},\rho) $.  We claim that
the derivative of $T_{u}$ at $(v,\ul{\delta},\rho)$ is an injection
for $(v,\ul{\delta},\rho)$ sufficiently small.  Indeed, the map
$T_{u}$ is the composition of the gluing construction for
pseudoholomorphic maps and the approximation construction for
$\rho$-vortices given above.  That the gluing construction has
injective derivative is proved in \cite{deform}, using the exponential
gluing profile so that the evaluation maps at the additional markings
are differentiable.  It then suffices to consider the case without
bubbles, for which we have
$$ 0 = \ddt |_{t = 0} \cF^{\rho} (T_{u}^\rho(v + t v')) =
\ti{D}_{A,u'} (D_{v} T^\rho_{u}(v')) \\ = \ti{D}_{A,u'} (v' +
   v'' )
$$
for some $v''$ in the image of the right inverse $Q^{0,\rho}$.  By
Lemma \ref{quadr}, the image of the right inverse has trivial
intersection with $\ker(\ti{D}_{{u}})$, which shows that $T_{u}$ is
an injection.  To show that $T_{u}$ is a homeomorphism, it remains to
show that $T_{u}$ is open, that is, if $[A_\nu,{u}_\nu,\rho_\nu] \to
[A,u',\rho]$ with $[A_\nu,{u}_\nu,\rho_\nu],[A,u',\rho]$ in the
image then $v_\nu \to v$ and $\ul{\delta}_\nu \to \ul{\delta}$.  The
injectivity argument in \cite{deform} shows that $\ul{\delta}_\nu \to
\ul{\delta}$.  By Theorem \ref{families2},
$T_{u}(v_\nu,\ul{\delta},\rho) \to T_{u}(v,\ul{\delta},\rho)$
implies $v_\nu \to v$.
\end{proof}  

\begin{proof}[Proof of Theorem \ref{cobord}]
That the moduli spaces $\ol{M}_n^K({\P^1},X)_{(0,\rho]}$ are
  stratified-smooth orbifolds for $\rho$ sufficiently small is
  \cite[Theorem 6.2.3]{cross}, using the fact that the $U(1)$-action
  on the moduli space of polarized vortices is free for $\rho$
  sufficiently small.  Charts for $\ol{M}_n^K({\P^1},X)_{[0,\rho]}$
  near zero-area vortices are given by the maps $T_{u}$.  By
  Proposition \ref{homeo}, $T_{u}$ is a homeomorphism onto an open
  neighborhood of $u$, and so defines a chart.  Independence of the
  choice of complement of $\aut(C)$ (given, for example, by choosing
  auxiliary hypersurfaces defining markings to make the domain curve
  stable) is similar to the case of pseudoholomorphic maps and left to
  the reader.  An orientation for the moduli space is constructed from
  orientations on the determinant lines of the linearized operators
  $\ti{D}_{A,u}$, induced from the deformation to the sum of
  $\d_A \oplus \d_A^*$ and a sum of linearized Cauchy-Riemann operators.  
\end{proof}

\subsection{Equality of invariants} 
Suppose that every zero-area vortex is regular.  In this
case, every element of $\ol{M}_n^{K,\fr}({\P^1},X,d)_{[0,\rho]}$ is stable
and regular for sufficiently small $\rho$, and so the projection
$ \ol{M}_n^{K,\fr}({\P^1},X,d)_{[0,\rho]} \to
\ol{M}_n^K({\P^1},X,d)_{[0,\rho]} $
has the structure of an orbifold principal $K^n$-bundle.  The
evaluation maps extend over the cobordism and the boundary
relation 
$ [\ol{M}_n^{K,}({\P^1},X,d)_{0}] =
[\ol{M}_n^{K}({\P^1},X,d)_{\rho}]$ in $
H(\ol{M}_n^{K}({\P^1},X,d)_{[0,\rho]},\Q) $
gives an equality of $\rho$ and $0$-vortex invariants.

\def\cprime{$'$} \def\cprime{$'$} \def\cprime{$'$} \def\cprime{$'$}
  \def\cprime{$'$} \def\cprime{$'$}
  \def\polhk#1{\setbox0=\hbox{#1}{\ooalign{\hidewidth
  \lower1.5ex\hbox{`}\hidewidth\crcr\unhbox0}}} \def\cprime{$'$}
  \def\cprime{$'$}

\end{document}